\documentclass[final]{lms}
\usepackage[psamsfonts]{amssymb}
\usepackage{amsmath}
\usepackage{enumerate}
\usepackage{epsfig}
\usepackage{leftidx}
\makeatletter
\gdef\@journal{%
  \vbox to 5.5\p@{\noindent
    \parbox[t]{\textwidth}{\raggedleft\normalfont\footnotesize\baselineskip 9\p@{\ \\}
  \vss}%
}}
\makeatother
\renewcommand\singlebox{\hskip1000pt minus 1fil\relax} 
\renewcommand\esinglebox{\hskip1000pt minus 1fil \llap{\proofbox}}
\newtheorem{theorem}{Theorem}[section] 
\newtheorem{lemma}[theorem]{Lemma}     
\newtheorem{corollary}[theorem]{Corollary}
\newtheorem{proposition}[theorem]{Proposition}

\newnumbered{assertion}{Assertion}    
\newnumbered{conjecture}{Conjecture}  
\newnumbered{definition}[theorem]{Definition}
\newnumbered{hypothesis}{Hypothesis}
\newnumbered{remark}[theorem]{Remark}
\newnumbered{note}{Note}
\newnumbered{observation}{Observation}
\newnumbered{problem}{Problem}
\newnumbered{question}{Question}
\newnumbered{algorithm}{Algorithm}
\newnumbered{example}[theorem]{Example}
\newunnumbered{notation}{Notation} 

\title[Energy measures, indices, and derivatives]%
{Energy measures and indices of Dirichlet forms,\\ with applications to derivatives on some fractals}
\author{Masanori Hino}

\classno{60J60 (primary), 28A80, 31C25, 60G44 (secondary)}

\extraline{Supported in part by the Ministry of Education, Culture, Sports, Science and Technology, Grant-in-Aid for Encouragement of Young Scientists, 18740070.}
\let\Bbb\mathbb 
\makeatletter
\def\rom#1{\mbox{\leavevmode\skip@\lastskip\unskip\/%
           \ifdim\skip@=\z@\else\hskip\skip@\fi{\rm{#1}}}}
\makeatother
\newcommand{\Thm}[1]{Theorem~\ref{th:#1}}
\newcommand{\Prop}[1]{Proposition~\ref{prop:#1}}
\newcommand{\Lem}[1]{Lemma~\ref{lem:#1}}

\newcommand{\Defn}[1]{Definition~\ref{def:#1}}

\newcommand{\Eq}[1]{\rom{(\ref{eq:#1})}}
\renewcommand{\a}{\alpha}
\newcommand{\gm}{\gamma}
\newcommand{\eps}{\varepsilon}\newcommand{\zt}{\zeta}

\newcommand{\lm}{\lambda}\newcommand{\kp}{\kappa}
\newcommand{\sg}{\sigma}
\newcommand{\ph}{\varphi}
\newcommand{\om}{\omega}

\newcommand{\Sg}{\Sigma}

\newcommand{\bfa}{{\boldsymbol a}}
\newcommand{\bfb}{{\boldsymbol b}}
\newcommand{\bfr}{{\boldsymbol r}}
\newcommand{\bfz}{{\boldsymbol 0}}
\newcommand{\bfone}{{\boldsymbol 1}}
\newcommand{\C}{{\Bbb C}}
\newcommand{\N}{{\Bbb N}}
\newcommand{\Q}{{\Bbb Q}}
\newcommand{\R}{{\Bbb R}}
\newcommand{\Z}{{\Bbb Z}}
\newcommand{\cB}{{\cal B}}
\newcommand{\cC}{{\cal C}}
\newcommand{\cD}{{\cal D}}
\newcommand{\cE}{{\cal E}}
\newcommand{\cF}{{\cal F}}

\newcommand{\cH}{{\cal H}}
\newcommand{\cI}{{\cal I}}
\newcommand{\cK}{{\cal K}}
\newcommand{\cL}{{\cal L}}
\newcommand{\cM}{{\cal M}}
\newcommand{\cN}{{\cal N}}
\newcommand{\cP}{{\cal P}}
\newcommand{\cS}{{\cal S}}
\newcommand{\cFz}{{\cal F}_{\text{\rm dom}}}
\newcommand{\la}{\langle}\newcommand{\ra}{\rangle}

\newcommand{\longto}{\longrightarrow}

\newcommand{\maruM}{{\stackrel{\circ}{\smash{\cal M}\rule{0pt}{1.3ex}}}}
\newcommand{\nab}{\nabla}

\newcommand{\rank}{\mathop{\mathrm{rank }}\nolimits}
\newcommand{\Id}{{\rm Id}}
\providecommand{\esssup}{\mathop{\rm ess\,sup}}
\providecommand{\supp}{\mathop{\rm supp}}
\providecommand{\Osc}{\mathop{\rm Osc}}
\providecommand{\Cp}{\mathop{\rm Cap}\nolimits}
\begin{document}
\maketitle
\begin{abstract}
We introduce the concept of index for regular Dirichlet forms by means of energy measures, and discuss its properties.
In particular, it is proved that the index of strong local regular Dirichlet forms is identical with the martingale dimension of the associated diffusion processes.
As an application, a class of self-similar fractals is taken up as an underlying space. We prove that first-order derivatives can be defined for functions in the domain of the Dirichlet forms and their total energies are represented as the square integrals of the derivatives.
\end{abstract}

\section{Introduction}
The concept of dimensions with regard to stochastic processes has been studied in various contexts. 
For example, Motoo and Watanabe~\cite{MW64} considered a class $\cM$ of martingale additive functionals  of general Markov processes, and proved that there exists a basis $\{x_n\}$ of $\cM$ such that every element in $\cM$ can be represented as a sum of stochastic integrals based on $\{x_n\}$ and a discontinuous part.
This is a broad extension of the study by Ventcel'~\cite{V61}, where the Brownian motion on $\R^d$ was considered.
The number of elements constituting the basis is sometimes called the {\em martingale dimension}, which coincides with the usual dimension of the underlying state space in typical cases.
Some related arguments are found in the papers by Kunita and Watanabe~\cite{KW67}, Cram\'er~\cite{Cr64}, and so on.
Later, as a variation, Davis and Varaiya~\cite{DV74} introduced  the concept of {\em multiplicity of filtration} in a more abstract setting, that is, on filtered probability spaces.

As another variation of these studies, in the first part of this article, we introduce the concept of {\em index} of regular Dirichlet forms $(\cE,\cF)$ using energy measures of functions in $\cF$. This definition is purely analytic.
We then prove some of its properties; among others, we describe the identification of the index with the martingale dimension of the  diffusion process associated with $(\cE,\cF)$ when the Dirichlet form is strong local.
This fact has been proved by Kusuoka~\cite{Ku89,Ku93} when the underlying space is a self-similar fractal (although the original definition of the index is equivalent to but different from that in this article).
Our result is regarded as a natural generalization of Kusuoka's one.

In the latter part of this article, we focus on the case that the underlying space $K$ is a fractal set.
In such a case, it is difficult to obtain an explicit value of the index (in other words, the martingale dimension) of canonical Dirichlet forms, unlike the case of spaces with differential structures.
Thus far, the only result that has been obtained is that Dirichlet forms $(\cE,\cF)$ associated with regular harmonic structures on post-critically finite (p.c.f.) self-similar fractals with some extra assumptions have an index of one~(\cite{Ku89,Hi08}).
As an application of this fact, we prove that every function $f$ in $\cF$ has a ``first-order derivative'' almost everywhere with respect to a minimal energy-dominant measure, and the total energy of $f$ (namely, $2\cE(f,f)$) is described as the square integral of the derivative. 
This representation appears as if the Dirichlet form were defined on a one-dimensional object, which can be regarded as a reflection of the property that the index is one.
Here, of course, derivatives in the usual sense do not exist since the underlying space $K$ does not have a differential structure. 
Instead, we take a good reference function $g$ from $\cF$ and define the derivative of $f\in\cF$ as the infinitesimal ratio of the fluctuation of two functions $f$ and $g$. 
From another viewpoint, we have a Taylor expansion of $f$ with respect to $g$ up to the first order.
The concept that the total energy of $f$ can be described as a square integral of a type of vector-valued gradient was first introduced by Kusuoka~\cite{Ku89,Ku93} (see also \cite{Te00} for a related work).
The contribution of this paper is the observation that only one good reference function is sufficient for the study of the infinitesimal behaviour of every function in $\cF$ and almost every point in $K$ when the index is one.
Recently, Pelander and Teplyaev~\cite{PT08} studied a topic similar to the one addressed in this article and proved the existence of derivatives of a type of smooth function on some self-similar fractals.
In their results, the existence of the derivative is assured almost everywhere with respect to self-similar measures, while in this paper, the derivative is proved to exist almost everywhere with respect to some energy measures. 
Since self-similar measures and energy measures are mutually singular in many cases~(\cite{Hi05,HN06}), these results are not comparable.

As a general theory, a good reference function can be taken from a certain dense subset of $\cF$. 
Exactly what function we can take is not obvious.
We will provide a partial answer to this problem; in the case of non-degenerate generalized Sierpinski gaskets, we prove that every nonconstant harmonic function can be taken as a reference function. As a byproduct of this result, the energy measures of arbitrary nonconstant harmonic functions are mutually absolutely continuous.
This fact is new even for the case of a two-dimensional standard Sierpinski gasket.

This article is organized as follows. 
In section~2, we introduce the concept of index of regular Dirichlet forms and study its properties.
In section~3, the index is characterized in terms of the martingale dimension when the Dirichlet form is strong local.
In section~4, we review the theory of self-similar Dirichlet forms on p.c.f.\ self-similar fractals; we then discuss the index of such fractals and related topics.
In the last section, we prove the existence of derivatives of functions on p.c.f.\ fractals when the index is one.
Some properties of Sierpinski gaskets are also investigated in detail.
\section{Concept of index of regular Dirichlet forms}
Let $K$ be a locally compact, separable, and metrizable space, and $\mu$, a $\sg$-finite Borel measure on $K$ with full support.
We denote the Borel $\sg$-field of $K$ by $\cB(K)$.
Let  $C(K)$ denote the set of all continuous real-valued functions on $K$, and $C_0(K)$, the set of all functions in $C(K)$ with compact support.
For $1\le p\le \infty$, $L^p(K,\mu)$ denotes the real $L^p$-space on the measure space $(K,\cB(K),\mu)$.
Suppose that we are given a symmetric regular Dirichlet form $(\cE,\cF)$ on $L^2(K,\mu)$.
That is, $\cF$ is a dense subspace of $L^2(K,\mu)$, $\cE\colon \cF\times\cF\to\R$ is a nonnegative definite symmetric bilinear form that satisfies the following:
\begin{itemize}
\item (Closedness) If a sequence $\{f_i\}_{i=1}^\infty$ in $\cF$ satisfies 
\[
\lim_{i\to\infty}\sup_{j>i}\cE_1(f_i-f_j,f_i-f_j)=0,
\]
then there exists $f\in\cF$ such that $\lim_{i\to\infty}\cE_1(f_i-f,f_i-f)=0$.
Here, we define 
\[
\cE_1(f,g):=\cE(f,g)+\int_K f(x)g(x)\,\mu(dx),\quad 
f,g\in\cF.
\]
\item (Markov property) For any $f\in\cF$, $\hat f:=\min(\max(0,f),1)$ belongs to $\cF$ and $\cE(\hat f,\hat f)\le \cE(f,f)$.
\item (Regularity) The space $\cF\cap C_0(K)$ is dense  in $\cF$ with norm $\|f\|_\cF:=\cE_1(f,f)^{1/2}$ and dense in $C_0(K)$ with uniform norm.\end{itemize}
The space $\cF$ becomes a Hilbert space under inner product $(f,g)_\cF:=\cE_1(f,g)$.
Hereafter, the topology of $\cF$ is always considered as that derived from norm $\|\cdot\|_\cF$.
We write $\cE(f)$ for $\cE(f,f)$ for simplicity.
For each $f\in\cF$, we can define a positive finite Borel measure $\nu_f$ on $K$ as follows~(\cite[section~3.2]{FOT}): When $f$ is bounded, $\nu_f$ is determined by the identity
\[
  \int_K \ph\,d\nu_f=2\cE(f\ph,f)-\cE(\ph,f^2)\quad
  \mbox{for all }\ph\in\cF\cap C_0(K).
\]
By utilizing the inequality
\begin{equation}\label{eq:energy}
  \left|\sqrt{\nu_f(B)}-\sqrt{\nu_g(B)}\right|^2\le \nu_{f-g}(B)
  \le 2\cE(f-g)
\end{equation}
for any Borel subset $B$ of $K$ and $f$, $g\in\cF\cap L^\infty(K,\mu)$
(cf.~\cite[p.~111, and (3.2.13) and (3.2.14) in p.~110]{FOT}), we can define $\nu_f$ for any $f\in\cF$ by a limiting procedure.
Then, equation \Eq{energy} still holds for any $f,g\in\cF$.
The measure $\nu_f$ is called the {\em energy measure} of $f$. (In the textbook~\cite{FOT} and the papers by the author~\cite{Hi05,HN06,Hi08}, the symbol $\mu_{\la f\ra}$ has been used to denote the energy measure of $f$.
In this paper, we use the symbol $\nu_f$ instead.)
For $f,g\in\cF$, the mutual energy measure $\nu_{f,g}$, which is a signed Borel measure on $K$, is defined as
\[
  \nu_{f,g}=\frac12(\nu_{f+g}-\nu_{f}-\nu_g).
\]
Then, $\nu_{f,f}=\nu_f$ and $\nu_{f,g}$ is bilinear in $f$, $g$ (\cite[p.~111]{FOT}).
The following inequality is also useful in the subsequent argument: for $f,g\in\cF$ and a Borel subset $B$ of $K$,
\begin{equation}
\label{eq:schwarz}
|\nu_{f,g}(B)|\le\sqrt{\nu_f(B)}\sqrt{\nu_g(B)}.
\end{equation}

For two $\sg$-finite (or signed) Borel measures $\mu_1$ and $\mu_2$ on $K$, we write $\mu_1\ll\mu_2$ if $\mu_1$ is absolutely continuous with respect to $\mu_2$, and $\mu_1\perp\mu_2$ if $\mu_1$ and $\mu_2$ are mutually singular.
The measure in the following definition is a natural generalization of the ones that were defined by Kusuoka~\cite{Ku89,Ku93}.
\begin{definition}\label{def:mdem}
A $\sg$-finite Borel measure $\nu$ on $K$ is called a {\em minimal energy-dominant measure} of $(\cE,\cF)$ if the following two conditions are satisfied.
\begin{enumerate}
\item (Domination) For every $f\in \cF$, $\nu_f\ll \nu$.
\item (Minimality) If another $\sg$-finite Borel measure $\nu'$ on $K$ satisfies condition (i) with $\nu$ replaced by $\nu'$, then $\nu\ll\nu'$.
\end{enumerate}
\end{definition}
By definition, two minimal energy-dominant measures are mutually absolutely continuous.
From \Eq{schwarz}, $\nu_{f,g}\ll\nu$ for a minimal energy-dominant measure $\nu$ and $f,g\in\cF$.
As observed from \Prop{dense} below, minimal energy-dominant measures can be realized by energy measures of some functions in $\cF$.
\begin{lemma}\label{lem:basic}
Let $\{f_i\}_{i=1}^\infty$ be a sequence of functions in $\cF$, $f\in\cF$, and let $\nu$ be a $\sg$-finite Borel measure on $K$. Suppose that $\nu_{f_i}\ll\nu$ for every $i\in\N$ and $f_i$ converges to $f$ in $\cF$ as $i\to\infty$. Then, $\nu_f\ll\nu$.
\end{lemma}
\begin{proof}
Suppose $\nu(B)=0$. Then, $\nu_{f_i}(B)=0$ for all $i$. From \Eq{energy},
\[
\nu_{f}(B)=\left|\sqrt{\nu_f(B)}-\sqrt{\nu_{f_i}(B)}\right|^2\le 2\cE(f-f_i)\to 0 \quad \mbox{as }i\to\infty.
\]
Therefore, $\nu_{f}(B)=0$. This implies the claim.
\end{proof}
\begin{lemma}\label{lem:mdem}
Let $\{f_i\}_{i=1}^\infty$ be a sequence of functions in $\cF$ such that the linear span of $\{f_i\}$ is dense in $\cF$.
Take a sequence $\{a_i\}_{i=1}^\infty$ of positive real numbers such that $\sum_{i=1}^\infty a_i \nu_{f_i}(K)$ converges.
Define $\nu=\sum_{i=1}^\infty a_i \nu_{f_i}$. Then, $\nu$ is a minimal energy-dominant measure.
\end{lemma}
\begin{proof}
If $f$ is a linear combination of $f_1,\dots,f_N$ for some $N\in\N$, then $\nu_f\ll\nu$ by using \Eq{schwarz}.
From \Lem{basic}, $\nu_f\ll\nu$ for all $f\in\cF$.
Therefore, assumption (i) in \Defn{mdem} holds.
Next, suppose that a $\sg$-finite Borel measure $\nu'$ on $K$ satisfies (i) in \Defn{mdem} with $\nu$ replaced by $\nu'$.
Then, $\nu_{f_i}\ll \nu'$ for any $i\in\N$. Since $\nu(B)=\sum_{i=1}^\infty a_i \nu_{f_i}(B)$ for $B\in\cB(K)$, we obtain $\nu\ll\nu'$. Therefore, assumption (ii) in \Defn{mdem} holds.
\end{proof}
\begin{lemma}\label{lem:null}
Let $\nu$ be a minimal energy-dominant measure and $B$, a Borel set of $K$.
Then, $\nu(B)=0$ if and only if $\nu_f(B)=0$ for every $f\in\cF$.
\end{lemma}
\begin{proof}
The only if part is obvious from the definition of the  minimal energy-dominant measure.
Suppose that $\nu_f(B)=0$ for every $f\in\cF$.
Take  $\{f_i\}_{i=1}^\infty$ and $\{a_i\}_{i=1}^\infty$ as given in \Lem{mdem} and $\nu'=\sum_{i=1}^\infty a_i \nu_{f_i}$.
Then, $\nu'$ is a minimal energy-dominant measure by \Lem{mdem}. 
Since  $\nu'$ and $\nu$ are mutually absolutely continuous and $\nu'(B)=0$, we obtain $\nu(B)=0$.
\end{proof}
For a signed Borel measure $\mu_1$ and a $\sg$-finite Borel measure $\mu_2$ on $K$ such that $\mu_1\ll\mu_2$, we denote the Radon--Nikodym derivative of $\mu_1$ with respect to $\mu_2$ by $\frac{d\mu_1}{d\mu_2}$ (or ${d\mu_1}/{d\mu_2}$).
\begin{lemma}\label{lem:energyineq}
Let $\nu$ be a minimal energy-dominant measure and $f,g\in\cF$.
\begin{enumerate}
\item It holds that $\left(\sqrt{{d\nu_f}/{d\nu}}-\sqrt{{d\nu_g}/{d\nu}}\right)^2\le {d\nu_{f-g}}/{d\nu}$ $\nu$-a.e.
In particular,
\begin{equation}\label{eq:sqrt}
 \int_K\left(\sqrt{\frac{d\nu_f}{d\nu}}-\sqrt{\frac{d\nu_g}{d\nu}}\right)^2\,d\nu
 \le \nu_{f-g}(K)\le 2\cE(f-g).
\end{equation}
\item Let $\{f_i\}_{i=1}^\infty$ and $\{g_i\}_{i=1}^\infty$ be sequences in $\cF$ and $f_i\to f$ and $g_i\to g$ in $\cF$ as $i\to\infty$.
Then, $d\nu_{f_i,g_i}/d\nu$ converges to $d\nu_{f,g}/d\nu$ in $L^1(K,\nu)$.
\end{enumerate}
\end{lemma}
\begin{proof}
(i) Since $d\nu_{sf-tg}/d\nu\ge0$ $\nu$-a.e.\ for all $s,t\in\R$, we have, for $\nu$-a.e., for all $s,t\in\Q$,
\[
  s^2\frac{d\nu_{f}}{d\nu}-2st\frac{d\nu_{f,g}}{d\nu}+t^2\frac{d\nu_{g}}{d\nu}\ge0.
\]
Therefore, 
\[
\left(\frac{d\nu_{f,g}}{d\nu}\right)^2\le \frac{d\nu_{f}}{d\nu}\cdot \frac{d\nu_{g}}{d\nu} \quad \nu\mbox{-a.e.}
\]
The assertions are derived from this inequality.

(ii) Since $\nu_{f,g}=(\nu_{f+g}-\nu_{f-g})/4$, it is sufficient to consider the case that $f_i=g_i$ for all $i$ and $f=g$.
From \Eq{sqrt}, we have
\begin{align}\label{eq:energyineq}
\int_K&\left|\frac{d\nu_{f_i}}{d\nu}-\frac{d\nu_f}{d\nu}\right|\,d\nu\nonumber\\
&=\int_K\left|\sqrt{\frac{d\nu_{f_i}}{d\nu}}-\sqrt{\frac{d\nu_f}{d\nu}}\right|
\left(\sqrt{\frac{d\nu_{f_i}}{d\nu}}+\sqrt{\frac{d\nu_f}{d\nu}}\right)
\,d\nu\nonumber\\
&\le\left\{\int_K\left(\sqrt{\frac{d\nu_{f_i}}{d\nu}}-\sqrt{\frac{d\nu_f}{d\nu}}\right)^2\,d\nu\right\}^{1/2}
\left\{\left(\int_K\frac{d\nu_{f_i}}{d\nu}\,d\nu\right)^{1/2}+\left(\int_K\frac{d\nu_f}{d\nu}\,d\nu\right)^{1/2}\right\}\nonumber\\
&\le \sqrt{2\cE(f_i-f)}\left(\sqrt{2\cE(f_i)}+\sqrt{2\cE(f)}\right),
\end{align}
which converges to $0$ as $i\to\infty$.
\end{proof}
Let $\{f_i\}_{i=1}^\infty$ be a sequence in $\cF$ and $\nu$, a minimal energy-dominant measure.
For $i,j\in\N$, denote the Radon--Nikodym derivative $d\nu_{f_i,f_j}/d\nu$ by $Z^{i,j}$.
When $\nu$ is a finite measure, one of the concrete ways to construct $Z^{i,j}$ is as follows. Let $\{\cB_n\}_{n=1}^\infty$ be a sequence of $\sg$-fields on $K$ such that $\cB_1\subset \cB_2\subset\cdots$, $\bigvee_{n=1}^\infty\cB_n=\cB(K)$, and each $\cB_n$ is generated by finitely many Borel subsets of $K$.
For each $n\in\N$, $\cB_n$ is provided by a partition of $K$ consisting of finitely many disjoint Borel sets $B_n^1,\dotsc,B_n^{M_n}$ for some $M_n\in \N$.
Then, for each $i,j\in\N$, the Radon--Nikodym derivative $Z_n^{i,j}$ of $\nu_{f_i,f_j}|_{\cB_n}$ with respect to $\nu|_{\cB_n}$ is defined by
\[
  Z_n^{i,j}(x)=\sum_{\a=1}^{M_n}\frac{\nu_{f_i,f_j}(B_n^\a)}{\nu(B_n^\a)}\cdot 1_{B_n^\a}(x),
  \quad x\in K.
\]
Here, $0/0:=1$ by convention.
We define
\begin{equation}\label{eq:K0}
  K_0=\{x\in K\mid \mbox{For every $i,j\in \N$, $\lim_{n\to\infty}Z_n^{i,j}(x)$ exists}\}.
\end{equation}
From the martingale convergence theorem, $\nu(K\setminus K_0)=0$.
For each $i,j\in \N$, we define
\begin{equation}\label{eq:Zij}
  Z^{i,j}(x)=\left\{
  \begin{array}{cl}
  \lim_{n\to\infty}Z_n^{i,j}(x) & \mbox{if }x\in K_0 \\
  0 & \mbox{if }x\in K\setminus K_0 
  \end{array}\right..
\end{equation}
Then, $Z^{i,j}$ is a Borel measurable function on $K$ and it is equal to $d\nu_{f_i,f_j}/d\nu$.
We thus have the following claim which is evident from the manner in which $Z^{i,j}(x)$ is constructed.
\begin{lemma}\label{lem:Z}
\begin{enumerate}
\item For every $i,j\in \N$ and $x\in K$, $|Z^{i,j}(x)|\le \sqrt{Z^{i,i}(x)}\sqrt{Z^{j,j}(x)}$.
\item For every $N\in \N$ and $x\in K$, the $N\times N$ matrix $\left(Z^{i,j}(x)\right)_{i,j=1}^N$ is symmetric and non\-negative definite.
\end{enumerate}
\end{lemma}
Even if $\nu$ is an infinite measure, it is evident to see that we can define a version of $Z^{i,j}=d\nu_{f_i,f_j}/d\nu$ such that the claims of \Lem{Z} hold.
Hereafter, we always take a version of the Radon--Nikodym derivatives $Z^{i,j}$ for given $\{f_i\}_{i=1}^\infty$ and $\nu$ in this manner.

Denote the usual real $\ell_2$-space by $\ell_2$, namely,
\[
  \ell_2=\{(a_i)_{i\in\N}\mid \mbox{$a_i\in\R$ for every $i\in\N$ and $\sum_{i=1}^\infty a_i^2<\infty$}\}.
\]
The canonical inner product and the norm on $\ell_2$ will be denoted by $(\cdot,\cdot)_{\ell_2}$ and $\|\cdot\|_{\ell_2}$, respectively.
We define
\begin{equation}\label{eq:Fzero}
  \cFz=\{f\in\cF\mid \mbox{$\nu_f$ is a minimal energy-dominant measure of }(\cE,\cF)\}.
\end{equation}
\begin{proposition}\label{prop:dense}
$\cFz$ is dense in $\cF$.
\end{proposition}
\begin{proof}
It is sufficient to prove that $d\nu_f/d\nu>0$ $\nu$-a.e.\ for $f$ in some dense subset of $\cF$, where $\nu$ is an arbitrarily fixed minimal energy-dominant measure.

Take a c.o.n.s.\ $\{f_i\}_{i=1}^\infty$ in $\cF$.
Define $\nu=\sum_{i=1}^\infty 2^{-i}\nu_{f_i}$. From \Lem{mdem}, $\nu$ is a minimal energy-dominant measure with $\nu(K)<\infty$.
We define $K_0$ and $Z^{i,j}$ for $i,j\in\N$ as \Eq{K0} and \Eq{Zij}, respectively.
Then, by construction, $\sum_{i=1}^\infty 2^{-i}Z^{i,i}(x)=1$ for $x\in K_0$.
In particular, the following holds:
\begin{equation}\label{eq:nondeg}
\mbox{For each $x\in K_0$, $Z^{i,i}(x)\ne0$ for some $i\in\N$}.
\end{equation}
Fix a Gaussian measure $\kp$ on $\ell_2$ such that the  support of $\kp$ is $\ell_2$ and $\kp$ does not charge any proper closed subspaces of $\ell_2$.

Let $\bfa=(a_i)_{i\in\N}\in\ell_2$ and define
\begin{align*}
g_N=\sum_{i=1}^N a_i 2^{-i/2}f_i \quad(N\in\N),\qquad
g=\sum_{i=1}^\infty a_i 2^{-i/2}f_i.
\end{align*}
Here, the infinite sum given above converges in $\cF$. 
Indeed,
\[
 \sum_{i=1}^\infty |a_i|2^{-i/2}\|f_i\|_\cF
 = \sum_{i=1}^\infty |a_i|2^{-i/2}
 \le\left(\sum_{i=1}^\infty a_i^2\right)^{1/2}
    \left(\sum_{i=1}^\infty 2^{-i}\right)^{1/2}
 =\|\bfa\|_{\ell_2}.
\]
We denote the map $\ell_2\ni\bfa\mapsto g\in\cF$ by $\Psi$, which is a contraction operator.
Then, from \Lem{energyineq}~(ii),
$d\nu_{g_N}/d\nu\,({}=\sum_{i,j=1}^N a_i a_j 2^{-(i+j)/2}Z^{i,j})$ converges to $d\nu_{g}/d\nu$ in $L^1(K,\nu)$.
On the other hand, for $x\in K$,
\begin{align}\label{eq:bounded}
\sum_{i,j=1}^N \left|a_i a_j 2^{-(i+j)/2}Z^{i,j}(x)\right|
&\le \sum_{i,j=1}^N |a_i||a_j|2^{-(i+j)/2}Z^{i,i}(x)^{1/2}Z^{j,j}(x)^{1/2} \quad \mbox{(from \Lem{Z}~(i))}\nonumber\\
&=\left(\sum_{i=1}^N |a_i| 2^{-i/2}Z^{i,i}(x)^{1/2}\right)^2\nonumber\\
&\le\left(\sum_{i=1}^N a_i^2\right)\left(\sum_{i=1}^N 2^{-i}Z^{i,i}(x)\right)\nonumber \\
&\le\sum_{i=1}^\infty a_i^2.
\end{align}
This implies that $(d\nu_{g_N}/d\nu)(x)$ converges as $N\to\infty$ for $x\in K$. The limit must be $(d\nu_{g}/d\nu)(x)$ for $\nu$-a.e.\,$x$.
We define $K_1=\{x\in K_0\mid \lim_{N\to\infty}(d\nu_{g_N}/d\nu)(x)=(d\nu_{g}/d\nu)(x)\}$.

Fix $x\in K$ and take $\bfa=(a_i)_{i\in\N}$ and $\bfb=(b_i)_{i\in\N}$ from $\ell_2$.
Then, $\sum_{i,j=1}^\infty a_i b_j 2^{-(i+j)/2}Z^{i,j}(x)$ converges absolutely.
Indeed, as in \Eq{bounded},
\begin{align*}
\sum_{i,j=1}^N \left|a_i b_j 2^{-(i+j)/2}Z^{i,j}(x)\right|
&\le\left(\sum_{i=1}^N |a_i|2^{-i/2}Z^{i,i}(x)^{1/2}\right)\left(\sum_{j=1}^N |b_j|2^{-j/2}Z^{j,j}(x)^{1/2}\right)\\
&\le\left(\sum_{i=1}^N a_i^2\right)^{1/2}\left(\sum_{j=1}^N b_j^2\right)^{1/2}\left(\sum_{i=1}^N 2^{-i}Z^{i,i}(x)\right)\\
&\le\|\bfa\|_{\ell_2}\|\bfb\|_{\ell_2}.
\end{align*}
From this domination, 
\[
\Phi_x(\bfa,\bfb):=\sum_{i,j=1}^\infty a_i b_j 2^{-(i+j)/2}Z^{i,j}(x),\quad
\bfa=(a_i)_{i\in\N}\in\ell_2,\ \bfb=(b_i)_{i\in\N}\in\ell_2
\]
provides a bounded nonnegative definite and symmetric bilinear form on $\ell_2$.
Then, there exists a bounded nonnegative definite and symmetric operator $A_x$ on $\ell_2$ such that
\[
  \Phi_x(\bfa,\bfb)=(\bfa,A_x\bfb)_{\ell_2},\quad
  \bfa,\bfb\in\ell_2.
\]
The kernel of $A_x$, which is denoted by $\ker A_x$, is equal to $\{\bfa\in\ell_2\mid \Phi_x(\bfa,\bfa)=0\}$.
When $x\in K_1$, $A_x\ne0$ from \Eq{nondeg}, which implies that $\ker A_x$ is a proper closed subspace of $\ell_2$.
Therefore, $\kp(\ker A_x)=0$ for $x\in K_1$, in particular, for $\nu$-a.e.\ $x$.

Now, we set
\begin{align*}
X&:=\{(x,\bfa)\in K\times\ell_2\mid \bfa\in \ker A_x\}\\
&=\left\{(x,\bfa)\in K\times\ell_2\left|\ \sum_{i,j=1}^\infty a_i a_j 2^{-(i+j)/2}Z^{i,j}(x)=0\quad(\bfa=(a_i)_{i\in\N})\right\}\right..
\end{align*}
This is a Borel subset of $K\times\ell_2$.
The above observation together with the Fubini theorem implies $(\nu\otimes\kp)(X)=0$.
Then, $\nu(X_\bfa)=0$ for $\kp$-a.e.\ $\bfa\in\ell_2$, where
$X_\bfa=\{x\in K\mid  (x,\bfa)\in X\}$.
  Consequently, there exists some $S\subset\ell_2$ with $\kp(\ell_2\setminus S)=0$ such that for $\bfa\in S$, $d\nu_g/d\nu>0$ $\nu$-a.e., where $g=\Psi(\bfa)$.
That is, $\nu_g$ is a minimal energy-dominant measure for such $g$. 
The map $\Psi\colon \ell_2\to\cF$ is contractive, and $\Psi(\ell_2)$ is dense in $\cF$.
Since $S$ is dense in $\ell_2$, $\Psi(S)$ is also dense in $\cF$. 
This completes the proof.
\end{proof}
\begin{remark}
In the proof of \Prop{dense}, it is not necessary for the measure $\kp$ to be Gaussian.
It is sufficient for the proof that $\kp$ has a full support and that it does not charge any proper closed subspaces of $\ell_2$. 
\end{remark}

Fix a minimal energy-dominant measure $\nu$ of $(\cE,\cF)$.
Let $\Z_+$ denote the set of all nonnegative integers.
\begin{definition}\label{def:index}
\begin{enumerate}
\item The {\em pointwise index} $p(x)$ of $(\cE,\cF)$ is defined as a measurable function on $K$ taking values in $\Z_+\cup\{+\infty\}$ such that the following hold:
\begin{enumerate}
\item For any $N\in\N$ and any $f_1,\dots,f_N\in\cF$, 
\[
 \rank \left(\frac{d\nu_{f_i,f_j}}{d\nu}(x)\right)_{i,j=1}^N\le p(x)\quad \mbox{for }\nu\mbox{-a.e.\,}x.
\]
\item If another function $p'(x)$ satisfies (a) with  $p(x)$ replaced by $p'(x)$, then $p(x)\le p'(x)$ $\nu$-a.e.\,$x$.
\end{enumerate}
\item The {\em index} $p$ of $(\cE,\cF)$ is defined as $p=\nu\mbox{-}\!\esssup_{x\in K}p(x)\in\Z_+\cup\{+\infty\}$.
In other words, $p$ is the smallest number satisfying the following:
for any $N\in\N$ and any $f_1,\dots,f_N\in\cF$, 
\[
 \rank \left(\frac{d\nu_{f_i,f_j}}{d\nu}(x)\right)_{i,j=1}^N\le p\quad \mbox{for }\nu\mbox{-a.e.\,}x.
\]

\end{enumerate}
\end{definition}
This definition is independent of the choice of a minimal energy-dominant measure $\nu$.
The pointwise index $p(x)$ is unique up to $\nu$-equivalence.
Its existence is assured by the following proposition.
\begin{proposition}\label{prop:rank}
Let $\{f_i\}_{i=1}^\infty$ be a sequence of functions in $\cF$ such that the linear span of $\{f_i\}_{i=1}^\infty$ is dense in $\cF$.
Denote the Radon--Nikodym derivative $d\nu_{f_i,f_j}/d\nu$ by $Z^{i,j}$ for $i,j\in\N$.
For each $N\in\N$ and $x\in K$, define an $N\times N$ matrix $Z_N(x)$ by 
\begin{equation}\label{eq:density}
Z_N(x)=\left(Z^{i,j}(x)\right)_{i,j=1}^N.
\end{equation}
Then, $p(x):=\sup_{N\in\N} \left(\rank Z_N(x)\right)$ is the pointwise index of $(\cE,\cF)$.
\end{proposition}
\begin{proof}
It is sufficient to prove that for any $M\in\N$ and $g_1,\dotsc,g_M\in\cF$, $\rank Y(x)\le p(x)$ for $\nu$-a.e.\,$x$, where $Y(x)=\bigl((d\nu_{g_k,g_l}/d\nu)(x)\bigr)_{k,l=1}^M$.

Denote the linear span of $\{f_i\}_{i=1}^\infty$ by $\widehat\cF$.
First, suppose that $g_k\in\widehat\cF$ for every $k$.
There exist some $\{a_{i,k}\}_{i=1,\dotsc,N,\,k=1,\dotsc,M}\subset\R$ for some $N\in\N$ such that $g_k=\sum_{i=1}^N a_{i,k}f_i$ for every $k=1,\dotsc,M$.
Let $C$ be an $N\times M$ matrix whose $(i,k)$-th component is $a_{i,k}$.
Then, for $k,l=1,\dotsc,N$,
\[
  \frac{d\nu_{g_k,g_l}}{d\nu}(x)=\sum_{i,j=1}^N a_{i,k}a_{j,l}Z^{i,j}(x),
\]
which is equal to the $(k,l)$-th component of the matrix $^t C Z_N(x)C$.
Therefore, $\rank Y(x)\le\rank Z_N(x)\le p(x)$ for $\nu$-a.e.\,$x$.

Next, suppose that $g_k\in\cF$ for every $k$.
Take $\{g_k^{(i)}\}_{k=1,\dotsc,M,\,i\in\N}$ from $\widehat\cF$ such that $\lim_{i\to\infty}g_k^{(i)}=g_k$ in $\cF$ for each $k$.
Then, for every $k$ and $l$, $\left.d\nu_{g_k^{(i)},g_l^{(i)}}\right/d\nu$ converges to $d\nu_{g_k,g_l}/d\nu$ in $L^1(K,\nu)$ as $i\to\infty$ from \Lem{energyineq}~(ii). 
By taking a subsequence if necessary, we may assume that this convergence is also in $\nu$-a.e.\ sense.
From the lower semi-continuity of $\rank$, we have
\[
  \rank Y(x)
  \le \liminf_{i\to\infty} \rank \left(\frac{d\nu_{g_k^{(i)},g_l^{(i)}}}{d\nu}(x)\right)_{k,l=1}^M
  \le p(x)
  \qquad \mbox{for }\nu\mbox{-a.e.\,}x.
\]
This completes the proof.
\end{proof}
\begin{proposition}\label{prop:zero}
The pointwise index $p(x)$ is greater than $0$ for $\nu$-a.e.\,$x$.
In particular, unless $\cE(f)=0$ for every $f\in\cF$, the index $p$ is greater than $0$.
\end{proposition}
\begin{proof}
Let $B=\{x\in K\mid p(x)=0\}$.
Take any $f\in\cF$.
Since the rank of the $1\times 1$ matrix $\left((d\nu_{f}/d\nu)(x)\right)$ is $0$ for $\nu$-a.e.\,$x$ in $B$, 
$d\nu_{f}/d\nu=0$ $\nu$-a.e.\ on $B$.
Therefore, $\nu_{f}(B)=0$.
From \Lem{null}, $\nu(B)=0$.
The latter assertion arises from the fact that $\nu$ does not vanish unless $\cE(f)=0$ for every $f\in\cF$.
\end{proof}
For $p\in\N$, let $(\cdot,\cdot)_{\R^p}$ and $|\cdot|_{\R^p}$ denote the standard inner product and the Euclidean norm on $\R^p$, respectively.
\begin{proposition}\label{prop:squareroot}
Let $p\in\N$. Let $\{f_i\}_{i=1}^\infty$ and $Z^{i,j}$ be the same as in \Prop{rank}.
Then, the index of $(\cE,\cF)$ is less than or equal to $p$ if and only if there exists a sequence $\{\zt^i\}_{i=1}^\infty$ of $\R^p$-valued measurable functions on $K$ such that for $\nu$-a.e.\,$x$, $\left(Z^{i,j}(x)\right)_{i,j=1}^\infty=\left((\zt^i(x),\zt^j(x))_{\R^p}\right)_{i,j=1}^\infty$, and $\zt^j_k(x)=0$ for all $j\in\N$ and $k>p(x)$, where $\zt^j(x)=(\zt^j_1(x),\dotsc,\zt^j_p(x))\in\R^p$ and $p(x)$ is the pointwise index of $(\cE,\cF)$.
\end{proposition}
\begin{proof*}
First, we prove the if part.
Let $C_N(x)$ denote the $p(x)\times N$ matrix whose $(i,j)$-th component is $\zt^j_i(x)$.
Then, the $N\times N$ matrix $(Z^{i,j}(x))_{i,j=1}^N$ is described as ${}^tC_N(x)C_N(x)$ for $\nu$-a.e.\,$x$, and thus its rank is less than or equal to $p(x)$.
Therefore, the index of $(\cE,\cF)$ does not exceed $p$.

Next, we prove the only if part.
We may assume that $p(x)$ is defined as in \Prop{rank}.
Let $r,N\in\N$ and 
\[
\cS_N^r=\{A\mid A\mbox{ is a nonnegative definite and real symmetric matrix of size $N$ and $\rank A\le r$}\}.
\]
For $A=(a^{i,j})_{i,j=1}^N\in \cS_N^r$, there exist $\xi^1,\dotsc,\xi^N\in\R^r$ such that $a^{i,j}=(\xi^i,\xi^j)_{\R^r}$ for every $i,j$.
Indeed, $A$ is decomposed as $A={}^tUDU$, where $U$ is an orthogonal matrix of size $N$ and
$D=(d^{i,j})_{i,j=1}^N$ is a real diagonal matrix of size $N$ such that $d^{i,i}\ge0$ for all $i$ and $d^{i,i}=0$ when $i>r$. 
Then, by letting $G=\sqrt{D}U$, we have $A={}^t GG$ and all the components of the $i$-th row of $G$ are zero for $i>r$.
Therefore, it suffices to define $\xi^j\in\R^r$ so that its $i$-th component is the $(i,j)$-th component of $G$.
Note also that $|\xi^j|_{\R^r}=\sqrt{a^{j,j}}$.

Now, let 
\[
  \cS_\infty^r=\left\{(a^{i,j})_{i,j=1}^\infty\left|\,
  \text{For every $N\in\N$, the matrix $A_N=(a^{i,j})_{i,j=1}^N$ belongs to $\cS_N^r$}\right\}\right.\!,
\]
which is regarded as a closed subset of $\R^{\N\times\N}$ (with the product topology).
For each $A=(a^{i,j})_{i,j=1}^\infty\in \cS_\infty^r$, let
\[
 \cK_A=\{(\xi^i)_{i\in\N}\mid
 \text{Each $\xi^i$ belongs to $\R^r$ and $a^{i,j}=(\xi^i,\xi^j)_{\R^r}$ for every $i$ and $j$}\}.
\]
We will prove $\cK_A\ne\emptyset$.
For each $N\in\N$, take $\xi_{(N)}^1,\dotsc,\xi_{(N)}^N\in\R^r$ such that $(\xi_{(N)}^i,\xi_{(N)}^j)_{\R^r}=a^{i,j}$ for all  $i,j=1,\dotsc,N$.
By the diagonalization argument, we can take an increasing sequence $\{N_k\}\uparrow\infty$ such that $\{\xi_{(N_k)}^i\}_{k=i}^\infty$ converges to some $\xi^i\in\R^r$ for all $i\in\N$.
Then, $(\xi^i)_{i\in\N}$ belongs to $\cK_A$.
In the same manner, we can prove that $\cK_A$ is a compact set of $(\R^r)^\N$ with the product topology.
Moreover, if a sequence $\{A_n\}_{n=1}^\infty$ in $\cS_\infty^r$ converges to some $A$ and $(\xi^{n,i})_{i\in\N}\in \cK_{A_n}$ for each $n$, then $\{(\xi^{n,i})_{i\in\N}\}_{n=1}^\infty$ has a limit point in $\cK_A$. 
Therefore, by the measurable selection theorem (see e.g.\ \cite[Lemma~12.1.8 and Theorem~12.1.10]{SV79}), there exists a Borel measurable map $\Theta_r\colon \cS_\infty^r\to(\R^r)^\N$ such that $\Theta_r(A)\in \cK_A$ for all $A\in \cS_\infty^r$.
For $r\le p$, define the map $\Xi_r\colon (\R^r)^\N\to(\R^p)^\N$ by
\[
  \Xi_r\left((\xi^j)_{j\in\N}\right)=\bigl((\xi^j_1,\dotsc,\xi^j_r,\underbrace{0,\dotsc,0}_{p-r})\bigr)_{j\in\N},
\]
where $\xi^j=(\xi^j_1,\dotsc,\xi^j_r)$.
Then, it suffices to define $\{\zt^i\}_{i=1}^\infty$ by
\[\singlebox
(\zt^i(x))_{i\in\N}=\Xi_{p(x)}\circ\Theta_{p(x)}\left((Z^{i,j}(x))_{i,j=1}^\infty\right), \quad x\in K.
\esinglebox\]
\end{proof*}
\begin{remark}
When $p=1$, the only if part of \Prop{squareroot} is proved in a simpler and more elementary manner as follows. 
We define $K(0)=\emptyset$ and
\[
  K(n):=\{x\in K\mid Z^{n,n}(x)>0\}\Bigm\backslash\bigcup_{i=0}^{n-1}K(i),\quad n\in\N,
\]
inductively.
It holds that $\nu\left(K\setminus\bigcup_{n=1}^\infty K(n)\right)=0$.
We define 
\[
\zt^i(x)=\sum_{n=1}^\infty 1_{K(n)}(x)\frac{Z^{i,n}(x)}{\sqrt{Z^{n,n}(x)}},\quad i\in\N.
\]
Then, when $x\in K(n)$ for $n\in\N$,
$\zt^i(x)\zt^j(x)=Z^{i,n}(x)Z^{j,n}(x)/Z^{n,n}(x)$, $i,j\in\N$.
Since 
\[
 \rank\begin{pmatrix} 
 Z^{i,i} & Z^{i,j} & Z^{i,n} \\
 Z^{j,i} & Z^{j,j} & Z^{j,n} \\
 Z^{n,i} & Z^{n,j} & Z^{n,n} 
 \end{pmatrix}(x)\le1 \quad \mbox{for }\nu\text{-a.e.\,}x,
\]
we have
\[
  0=\det\begin{pmatrix}
  Z^{i,j} & Z^{i,n}\\
  Z^{n,j} & Z^{n,n}
  \end{pmatrix}(x)
  =Z^{i,j}(x)Z^{n,n}(x)-Z^{i,n}(x)Z^{j,n}(x)
  \quad \mbox{for }\nu\text{-a.e.\,}x.
\]
Therefore, $\zt^i(x)\zt^j(x)=Z^{i,j}(x)$ for $\nu\text{-a.e.\,}x$.
\end{remark}
We will discuss the stability of the pointwise index.
Recall that a regular Dirichlet form $(\cE,\cF)$ on $L^2(K,\mu)$ is called {\em strong local} if $\cE(f,g)=0$ for any $f,g\in\cF$ as long as both $\supp f$ and $\supp g$ are compact and $g$ is constant on a neighborhood of $\supp f$. Here, $\supp f$ is defined as the  support of the measure $|f|\cdot\mu$ on $K$.
\begin{proposition}\label{prop:stability}
  Suppose that $(\cE,\cF)$ and $(\tilde\cE,\tilde\cF)$ are both strong local regular Dirichlet forms on $L^2(K,\mu)$.
  If these are equivalent, namely, $\cF=\tilde\cF$ and there exist positive constants $c_1$ and $c_2$ such that
\[
  c_1\cE(f)\le \tilde\cE(f)\le c_2\cE(f)\quad
  \mbox{for all } f\in\cF,
\]
then $\nu$ is also a minimal energy-dominant measure of $(\tilde \cE,\tilde\cF)$, and the pointwise indices of $(\cE,\cF)$ and $(\tilde\cE,\tilde\cF)$ coincide for $\nu$-a.e.\ $x$.
\end{proposition}
\begin{proof}
Let $\tilde\nu_f$ denote the energy measure of $f$ with respect to $(\tilde\cE,\tilde\cF)$.
From \cite[Proposition~1.5.5~(b)]{Le78} (or \cite[p.\,389]{Mo94}),
\begin{equation}\label{eq:equivalence}
  c_1^2\nu_f\le\tilde\nu_f\le c_2^2 \nu_f,\quad f\in\cF.
\end{equation}
Therefore, $\nu$ is also a minimal energy-dominant measure of $(\tilde\cE,\tilde\cF)$.

Take any $f_1,\dotsc,f_N\in\cF$, $N\in\N$.
We define $Z_N(x)=\left((d\nu_{f_i,f_j}/d\nu)(x)\right)_{i,j=1}^N$ and $\tilde Z_N(x)=\left((d\tilde\nu_{f_i,f_j}/d\nu)(x)\right)_{i,j=1}^N$.
For $n=1,\dotsc,N$, denote the $n$-th eigenvalue of $Z_N(x)$ (resp.\ $\tilde Z_N(x)$) from below by $\lm_n(x)$ (resp.\ $\tilde\lm_n(x)$).
From the minmax principle,
\[
  \lm_n(x)=\inf_{M\in G_{N,n}(\R)}\left(
  \sup_{\bfa\in M\setminus\{\bfz\}}\frac{^t\bfa Z_N(x)\bfa}{|\bfa|_{\R^N}^2}\right)\mbox{~~and~~}
  \tilde\lm_n(x)=\inf_{M\in G_{N,n}(\R)}\left(
  \sup_{\bfa\in M\setminus\{\bfz\}}\frac{^t\bfa \tilde Z_N(x)\bfa}{|\bfa|_{\R^N}^2}\right).
\]
Here, $G_{N,n}(\R)$ is the Grassmann manifold consisting of all $n$-dimensional subspaces of $\R^N$.
Note that by the separability of $G_{N,n}(\R)$, we may replace $G_{N,n}(\R)$ and $M\setminus\{\bfz\}$ in the above equations by countable dense subsets.
Since 
\[
^t\bfa Z_N(x)\bfa=\frac{d\nu_{a_1f_1+\dotsb+a_Nf_N}}{d\nu}(x)
\mbox{~~and~~}
 {}^t\bfa \tilde Z_N(x)\bfa=\frac{d\tilde\nu_{a_1f_1+\dotsb+a_Nf_N}}{d\nu}(x)
\]
for $\bfa={}^t(a_1,\dotsc,a_N)\in\R^N$, from \Eq{equivalence},
\[
  c_1^2\,^t\bfa Z_N(x)\bfa\le {}^t\bfa \tilde Z_N(x)\bfa
  \le c_2^2\,^t\bfa Z_N(x)\bfa
  \quad \mbox{for }\nu\mbox{-a.e.\,}x.
\]
Therefore, $c_1^2\lm_n(x)\le \tilde\lm_n(x)\le c_2^2\lm_n(x)$ for $\nu$-a.e.\,$x$.
This implies that $\rank Z_N(x)=\rank \tilde Z_N(x)$ for $\nu$-a.e.\,$x$, which completes the proof.
\end{proof}
\begin{example}
Let $K=\R^d$ and $\mu$ be the Lebesgue measure $dx$.
Define a Dirichlet form $(\cE,\cF)$ on $L^2(\R^d,dx)$ as
\[
  \cE(f,g)=\frac12\int_{\R^d}(\nab f,\nab g)_{\R^d}\,dx,\quad
  f,g\in\cF=H^1(\R^d),
\]
where $H^1(\R^d)$ is the first-order $L^2$-Sobolev space on $\R^d$.
Then, it is evident that
$\nu_{f,g}(dx)=(\nab f(x),\nab g(x))_{\R^d}\,dx$ for $f,g\in\cF$, and the Lebesgue measure can be taken as a minimal energy-dominant measure.
From \Prop{squareroot}, the index is less than or equal to $d$.
For $R>0$, let $\ph(t)$ be a $C^\infty$-function on $\R$ with compact support such that 
$\ph(t)=1$ on $[-R,R]$ and $\ph(t)=0$ on $\R\setminus[-R-1,R+1]$.
Take $f_i(x)=x_i\ph(|x|_{\R^d})$, $i=1,\dotsc,d$, $x=(x_1,\dotsc,x_d)\in\R^d$.
Then, $f_i\in\cF$ and
\[
\frac{d\nu_{f_i,f_j}}{dx}(x)=\begin{cases}1&(i=j)\\0&(i\ne j)\end{cases}
\quad\text{if }|x|_{\R^d}<R.
\]
Therefore, $\rank\left((d\nu_{f_i,f_j}/dx)(x)\right)_{i,j=1}^d=d$ when $|x|_{\R^d}<R$.
Since $R$ is arbitrary, the pointwise index is $d$ $dx$-a.e.\ and the index is $d$.
From \Prop{stability}, the same is true for the Dirichlet form $(\cE',\cF')$ on $L^2(\R^d,dx)$ defined by
\[
\cE'(f,g)=\frac12\int_{\R^d}(\sg(x)\nabla f(x),\nabla g(x))_{\R^d}\,dx,\quad f,g\in\cF'=H^1(\R^d),
\]
where $\sg(x)$ is a $d\times d$ matrix valued measurable function on $\R^d$ such that there exist some positive numbers $c_3$ and $c_4$ satisfying
\[
  c_3|h|_{\R^d}^2\le (\sg(x)h,h)_{\R^d}\le c_4|h|_{\R^d}^2,\quad h\in\R^d,\ x\in\R^d.
\]
When $\sg(x)$ is degenerate or unbounded, the pointwise index $p(x)$ should be equal to  $\rank\sg(x)$ $dx$-a.e.\ as long as the domain of the Dirichlet form contains sufficiently many functions so that the argument similar to the above one is valid.
\end{example}
In the example above, the index can be calculated  easily because the Dirichlet form is given by the square integral of the gradient.
Otherwise, determining the index is not straightforward and it may be difficult to determine.
For instance, it is an open problem to determine the index of the canonical Dirichlet forms on Sierpinski carpets, which are typical infinitely ramified self-similar fractals.
\section{Probabilistic counterpart of index}
In this section, we discuss the probabilistic interpretation of the index of $(\cE,\cF)$.
For this purpose, let us review the theory of additive functionals associated with $(\cE,\cF)$ on $L^2(K,\mu)$, following \cite[Chapter~5]{FOT}.
The capacity $\Cp$ associated with $(\cE,\cF)$ is defined as
\[
  \Cp(U)=\inf\{\cE_1(f)\mid f\in\cF\mbox{ and }f\ge1\ \mu\mbox{-a.e.\ on }U\}
\]
if $U$ is an open subset of $K$, and
\[
  \Cp(B)=\inf\{\Cp(U)\mid \mbox{$U$ is open and $U\supset B$}\}
\]
for general subsets $B$ of $K$.
A set $B\subset K$ with $\Cp(B)=0$ is called an {\em exceptional set}.
A statement depending on $x\in K$ is said to hold for {\em q.e.} (quasi-every) $x$ if the set of $x$ for which the statement is not true is an exceptional set. 

In what follows, we consider only the case that $(\cE,\cF)$ is strong local.
  From the general theory of regular Dirichlet forms, we can construct a diffusion process $\{X_t\}$ on $K_\Delta$ defined on a filtered probability space $(\Omega,\cF_\infty,P,\{P_x\}_{x\in K_\Delta},\{\cF_t\}_{t\in[0,\infty)})$ associated with $(\cE,\cF)$.
Here, $K_\Delta=K\cup\{\Delta\}$ is a one-point compactification of $K$ and $\{\cF_t\}_{t\in[0,\infty)}$ is a minimum completed admissible filtration.
Any numerical function $f$ on $K$ extends to $K_\Delta$ by letting $f(\Delta)=0$.
The relationship between $\{X_t\}$ and $(\cE,\cF)$ is explained in such a way that the operator $f\mapsto E_x[f(X_t)]$ produces the semigroup associated with $(\cE,\cF)$, where $E_x$ denote the expectation with respect to $P_x$.
  We may assume that for each $t\in[0,\infty)$, there exists a shift operator $\theta_t\colon \Omega\to\Omega$ that satisfies $X_s\circ\theta_t=X_{s+t}$ for all $s\ge0$.
  Denote the life time of $\{X_t(\om)\}_{t\in[0,\infty)}$ by $\zt(\om)$.
  A $[-\infty,+\infty]$-valued function $A_t(\om)$, $t\in[0,\infty)$, $\om\in\Omega$, is referred to as an {\em additive functional} if the following conditions hold:
  \begin{itemize}
  \item $A_t(\cdot)$ is $\cF_t$-measurable for each $t\ge0$.
  \item There exist a set $\Lambda\in\sigma(\cF_t; t\ge0)$ and an exceptional set $N\subset K$ such that $P_x(\Lambda)=1$ for all $x\in K\setminus N$ and $\theta_t\Lambda\subset\Lambda$ for all $t>0$; moreover, for each $\om\in\Lambda$, $A_\cdot(\om)$ is right continuous and has the left limit on $[0,\zt(\om))$, $A_0(\om)=0$, $|A_t(\om)|<\infty$ for all $t<\zt(\om)$, $A_t(\om)=A_{\zt(\om)}(\om)$ for $t\ge\zt(\om)$, and
  \[
    A_{t+s}(\om)=A_s(\om)+A_t(\theta_s \om)
    \quad \mbox{for every }t,s\ge0.
  \]
  \end{itemize}
  The sets $\Lambda$ and $N$ referred to above are called a {\em defining set} and an {\em exceptional set} of the additive functional $A$, respectively.
  A finite (resp.\ continuous) additive functional is defined as an additive functional such that $|A_\cdot(\om)|<\infty$ (resp.\ $A_\cdot(\om)$ is continuous) on $[0,\infty)$ for $\om\in\Lambda$.
  A $[0,\infty]$-valued continuous additive functional is referred to as a positive continuous additive functional.
   From \cite[Theorems~5.1.3 and 5.1.4]{FOT}, for each positive continuous additive functional $A$, there exists a unique measure~$\mu_A$ on $K$ (termed the Revuz measure of $A$) such that the following identity holds for any $t>0$ and nonnegative Borel functions $f$ and $h$ on $K$:
  \begin{align}\label{eq:Revuz}
  \int_K E_x\left[\int_0^t f(X_s)\,dA_s\right]h(x)\,\mu(dx)
  =\int_0^t \int_K E_x\left[h(X_s)\right]f(x)\,\mu_A(dx)\,ds.\nonumber
  \end{align}
  Further, if two positive continuous additive functionals $A^{(1)}$ and $A^{(2)}$ have the same Revuz measures, then $A^{(1)}$ and $A^{(2)}$ coincide up to the natural equivalence.
  
  Let $P_\mu$ be a measure on $\Omega$ defined as $P_\mu(\cdot)=\int_K P_x(\cdot)\,\mu(dx)$.
  Let $E_\mu$ denote the integration with respect to $P_\mu$.
  We define the energy $e(A)$ of additive functional $A$ as
  \[
  e(A)=\lim_{t\to0}(2t)^{-1}E_\mu[A_t^2]
  \]
  if the limit exists.
  
  Let $\cM$ be the space of martingale additive functionals of $\{X_t\}$ that is defined as
  \[
  \cM=\left\{M\left|\begin{array}{ll}
  \text{$M$ is a finite additive functional such that $M_\cdot(\om)$ is right continuous}\\
  \text{and has a left limit on $[0,\infty)$ for $\om$ in a defining set of $M$, and for}\\
  \text{each $t>0$, $E_x[M_t^2]<\infty$ and $E_x[M_t]=0$ for q.e.\  $x\in K$}\end{array}\!\!
  \right.\right\}.
  \]
  Due to the assumption that $(\cE,\cF)$ is strong local, any $M\in\cM$ is in fact a continuous additive functional (\cite[Lemma~5.5.1~(ii)]{FOT}).
  
  Each $M\in\cM$ admits a positive continuous additive functional $\la M\ra$ referred to as the quadratic variation associated with $M$ that satisfies
  \[
  E_x[\la M\ra_t]=E_x[M_t^2], \ t>0\text{ for q.e.\ $x\in K$},
  \] 
  and the following equation holds:
  \begin{equation}\label{eq:energyidentity}
  e(M)=\frac12 \mu_{\la M\ra}(K). 
  \end{equation}
  We set $\maruM=\{M\in\cM\mid e(M)<\infty\}$.
  Then, $\maruM$ is a Hilbert space with inner product $e(M,M'):=(e(M+M')-e(M)-e(M'))/2$ (\cite[Theorem~5.2.1]{FOT}). 
  For $M,L\in \maruM$, we set $\mu_{\la M,L\ra}=(\mu_{\la M+L\ra}-\mu_{\la M\ra}-\mu_{\la L\ra})/2$.
  Since $\mu_{\la M,L\ra}$ is bilinear and symmetric with respect to $M,L$ and $\mu_{\la M,M\ra}=\mu_{\la M\ra}$ is a positive measure, for any nonnegative function $f$ in $L^1(K,\mu_{\la M\ra}+\mu_{\la L\ra})$, it holds that
  \begin{equation}\label{eq:schwarzM}
  \left|\int_K f\,d\mu_{\la M,L\ra}\right|
  \le \sqrt{\int_K f\,d\mu_{\la M\ra}}
      \sqrt{\int_K f\,d\mu_{\la L\ra}}
  \end{equation}
and
  \begin{equation}\label{eq:minkowskiM}
  \left|\sqrt{\int_K f\,d\mu_{\la M\ra}}
  -\sqrt{\int_K f\,d\mu_{\la L\ra}}\right|^2
  \le\int_K f\,d\mu_{\la M-L\ra}.      
  \end{equation}
  For $M\in\maruM$ and $f\in L^2(K,\mu_{\la M\ra})$, we can define the stochastic integral $f\bullet M$ (\cite[Theorem~5.6.1]{FOT}), which is a unique element in $\maruM$ such that
  \[
  e(f\bullet M,L)=\frac12\int_Kf(x)\mu_{\la M,L\ra}(dx)\quad
  \text{for all }L\in\maruM.
  \]
  We may write $\int_0^\cdot f(X_t)\,dM_t$ for $f\bullet M$ since $(f\bullet M)_t=\int_0^t f(X_s)\,dM_s$, $t>0$, $P_x$-a.e.\ for q.e.\,$x\in K$ as long as $f$ is a continuous function with compact support on $K$ (\cite[Lemma~5.6.2]{FOT}).
  From \cite[Lemma~5.6.2]{FOT}, we also have
  \begin{align}\label{eq:integral}
  d\mu_{\la f\bullet M,L\ra}=f\cdot d\mu_{\la M,L\ra},\quad
  L\in\maruM.
  \end{align}
The space $\cN_c$ of the continuous additive functionals of zero energy is defined as
  \[
  \cN_c=\left\{N\left|
  \begin{array}{l}\text{$N$ is a continuous additive functional},\\
  \text{$e(N)=0$, $E_x[|N_t|]<\infty$ for q.e.\ $x\in K$ and $t>0$}
  \end{array}\!\right\}\right..
  \]
  For each $u\in\cF$, we have a quasi-continuous modification $\tilde u$ of $u$ in the restricted sense, that is, $u=\tilde u$ $\mu$-a.e., and for any $\eps>0$, there exists an open subset $G$ of $K$ such that $\Cp(K\setminus G)<\eps$ and $\tilde u|_{K_\Delta\setminus G}$ is continuous (\cite[Theorem~2.1.3]{FOT}).  
  Then, the Fukushima decomposition theorem~\cite[Theorem~5.2.2]{FOT} says that there exist unique  $M^{[u]}\in\maruM$ and $N^{[u]}\in\cN_c$ such that
  \[
  \tilde u(X_t)-\tilde u(X_0)=M_t^{[u]}+N_t^{[u]},
  \quad t>0,\ P_x\mbox{-a.e.}\mbox{ for q.e.\,}x.
    \]
  From \cite[Theorem~5.2.3]{FOT}, $\mu_{\la M^{[u]}\ra}$ is equal to the energy measure $\nu_u$ of $u$.
  Therefore, 
\begin{equation}\label{eq:MAF}
\mu_{\la M^{[u]},M^{[v]}\ra}=\nu_{u,v},\quad u,v\in\cF.
\end{equation} 
We recall the following claim.
\begin{lemma}[{(\cite[Lemma~5.6.3]{FOT})}]\label{lem:density}
Let $C_1$ be a dense subset of $C_0(K)$ with uniform norm and $\cF_1$, a dense subset of $\cF$.
Then, $\{u\bullet M^{[v]}\mid u\in C_1,\ v\in \cF_1\}$ is dense in $(\maruM,e)$.
\end{lemma}
By using this lemma, we can prove some basic properties for Revuz measures as follows.
\begin{lemma}\label{lem:M}
Let $\nu$ be a minimal energy-dominant measure of $(\cE,\cF)$, and let $M,M'\in\maruM$.
Then, the following hold.
\begin{enumerate}
\item $\mu_{\la M,M'\ra}$ is absolutely continuous with respect to $\nu$,
\begin{equation}\label{eq:schwarzMM}
  \left|\frac{d\mu_{\la M,M'\ra}}{d\nu}\right|
  \le \sqrt{\frac{d\mu_{\la M\ra}}{d\nu}}
      \sqrt{\frac{d\mu_{\la M'\ra}}{d\nu}}
  \quad \nu\mbox{-a.e.},
\end{equation}
and
\begin{equation}\label{eq:sqrtMM}
\int_K\left(\sqrt{\frac{d\mu_{\la M\ra}}{d\nu}}
     - \sqrt{\frac{d\mu_{\la M'\ra}}{d\nu}}\right)^2\,d\nu
\le\mu_{\la M-M'\ra}(K) =2e(M-M').
\end{equation}
\item If $M_n\to M$ and $M'_n\to M'$ in $(\maruM,e)$ as $n\to\infty$, then
\[
  \frac{d\mu_{\la M_n,M'_n\ra}}{d\nu}\to \frac{d\mu_{\la M,M'\ra}}{d\nu} \mbox{ in $L^1(K,\nu)$ as $n\to\infty$}.
\]
\end{enumerate}
\end{lemma}
\begin{proof}
(i): When $M=u\bullet M^{[v]}$ for $u\in C_0(K)$ and $v\in \cF$, $\mu_{\la M\ra}=u^2\cdot\nu_v$, which is absolutely continuous with respect to $\nu$. 
From \Lem{density}, \Eq{minkowskiM}, and \Eq{energyidentity}, $\mu_{\la M\ra}\ll\nu$ for all $M\in\maruM$. 
In addition, $\mu_{\la M,M'\ra}\ll\nu$ from \Eq{schwarzM}. 
Equations~\Eq{schwarzMM} and \Eq{sqrtMM} follow from an argument similar to the proof of \Lem{energyineq}~(i).

(ii): This claim is proved in exactly the same manner as \Lem{energyineq}~(ii).
\end{proof}
\begin{definition}[(cf.~\cite{Hi08})]
The {\em AF-martingale dimension} of $\{X_t\}$ (or of $(\cE,\cF)$) is defined as a smallest number $p$ in $\Z_+$ satisfying the following: there exists a sequence $\{M^{(k)}\}_{k=1}^p$ in $\maruM$ such that every $M\in\maruM$ has a stochastic integral representation
\[
  M_t=\sum_{k=1}^p(h_k\bullet M^{(k)})_t,\quad t>0,\ P_x\mbox{-a.e.\ for q.e.\,}x,
\]
where $h_k\in L^2(K,\mu_{\la M^{(k)}\ra})$ for each $k=1,\dotsc,p$.
If such $p$ does not exist, the AF-martingale dimension is defined as $+\infty$.
\end{definition}
This definition is basically consistent with the works by Motoo and Watanabe~\cite{MW64} and Davis and Varaiya~\cite{DV74}.
From now on, we will omit the symbol AF (an abbreviation of ``additive functionals'') and only write martingale dimension.
\begin{theorem}\label{th:index}
The index of $(\cE,\cF)$ coincides with the martingale dimension of $\{X_t\}$.
\end{theorem}
This theorem is a natural generalization of  \cite[Theorem~6.12]{Ku93}.
The remainder of this section describes the proof of this theorem. 

Let 
\[
\ell_0=\{a=(a_i)_{i\in\N}\mid a_i\in\R\mbox{ for all }i\in\N \mbox{ and $a_i=0$ except for finitely many $i$}\}.
\]
Fix a minimal energy-dominant measure $\nu$ with $\nu(K)<\infty$ and a sequence $\{f_i\}_{i=1}^\infty$ in $\cF$ such that the linear span of $\{f_i\}_{i=1}^\infty$ is dense in $\cF$.
Denote the Radon--Nikodym derivative $d\nu_{f_i,f_j}/d\nu$ by $Z^{i,j}$ for $i,j\in\N$.
 For each $x\in K$, define pre-inner product $\la\cdot,\cdot\ra_{Z(x)}$ on $\ell_0$ as
 \[
 \la a,b\ra_{Z(x)}:=\sum_{i,j=1}^\infty a_i b_j Z^{i,j}(x),\quad
 a=(a_i)_{i\in\N}\in \ell_0,\ b=(b_j)_{j\in\N}\in \ell_0.
 \]
 For $x\in K$, let $\ell_{Z(x)}$ be the set of all equivalent classes of $\ell_0$ with respect to the equivalent relation~$\sim_x$ derived from pre-inner product $\la\cdot,\cdot\ra_{Z(x)}$.
 That is, $\ell_{Z(x)}=\ell_0/\!\sim_x$,
 where $a\sim_x b$ if and only if $\la a-b,a-b\ra_{Z(x)}=0$.
 \begin{lemma}\label{lem:dimension}
 The dimension of $\ell_{Z(x)}$ is equal to the pointwise index $p(x)$ of $(\cE,\cF)$ for $\nu$-a.e.\,$x$.
\end{lemma}
\begin{proof}
Fix $x\in K$ and take $N\in\N$.
Let $Z_N(x)=\left(Z^{i,j}(x)\right)_{i,j=1}^N$.
Then, it is an elementary fact that 
\[
  \dim\left(\R^N/\!\sim_{x,N}\right)=\rank Z_N(x),
\]
where $a\sim_{x,N}b$ is defined as $\sum_{i,j=1}^N a_ib_j Z^{i,j}(x)=0$ for $a=(a_i)_{i=1}^N\in \R^N$ and $b=(b_j)_{j=1}^N\in \R^N$.
By letting $N\to\infty$, we obtain $\dim \ell_{Z(x)}=p(x)$ for $\nu$-a.e.\,$x$ from \Prop{rank}.
\end{proof}
Let
\[
\cC=\left\{g=(g_i)_{i\in\N}\left|\begin{array}{l}
\text{Each $g_i$ is a bounded Borel function on $K$, and there exists}\\
\text{some $n\in\N$ such that $g_i=0$ for all $i\ge n$}
\end{array}\right.\!\!
\right\}.
\]
Note that for $g\in\cC$, $g(x)=(g_i(x))_{i\in\N}$ belongs to $\ell_0$ for each $x\in K$.
We define a pre-inner product $\la\cdot,\cdot\ra_Z$ on $\cC$ by
\[
\la g,g'\ra_{Z}=\frac12\int_K \la g(x),g'(x)\ra_{Z(x)}\,\nu(dx),
\quad g,g'\in\cC.
\]
  For $g=(g_i)_{i\in\N}\in\cC$, we define 
\begin{align}\label{eq:chi}
\chi(g)=\sum_{i=1}^\infty g_i\bullet M^{[f_i]}\in\maruM.
\end{align}
  We note that the sum above is in fact a finite sum.
\begin{lemma}\label{lem:isometry}
The map $\chi\colon \cC\to\maruM$ preserves the \mbox{(pre-)inner} products.
\end{lemma}
\begin{proof}
 For $g=(g_i)_{i\in\N}\in\cC$ and $g'=(g'_i)_{i\in\N}\in\cC$, from \Eq{energyidentity}, \Eq{integral}, and \Eq{MAF},
  \begin{align*}
  e(\chi(g),\chi(g'))
  &=\frac12\sum_{i,j}\int_K g_i(x)g_j'(x)\,\mu_{\la M^{[f_i]},M^{[f_j]}\ra}(dx)\\
  &=\frac12\sum_{i,j}\int_K g_i(x) g'_j(x) Z^{i,j}(x)\,\nu(dx)\\
  &=\la g,g'\ra_{Z}.
  \end{align*}
This completes the proof.
\end{proof}

By virtue of \Lem{density} and the fact that the linear span of $\{f_i\}_{i=1}^\infty$ is dense in $\cF$, $\chi(\cC)$ is dense in $\maruM$.
Define
\[
 \cD=\{g=(g_i)_{i\in\N}\mid
\text{Each $g_i$ is a Borel function on $K$, and $(g_i(x))_{i\in\N}\in\ell_0$ for $\nu$-a.e.\,$x$}\}.
\]
\begin{lemma}\label{lem:onb}
Suppose that the index $p$ of $(\cE,\cF)$ is finite and nonzero.
Then, there exist $g^{(k)}=(g_i^{(k)})_{i\in\N}\in\cD$, $k=1,\dotsc,p$ such that
\begin{align*}
\la g^{(k)}(x),g^{(l)}(x)\ra_{Z(x)}&=0 \quad \nu\mbox{-a.e.\,$x$ for $k\ne l$}\\
\intertext{and}
\la g^{(k)}(x),g^{(k)}(x)\ra_{Z(x)}&=\left\{\begin{array}{cl}
1& \mbox{when }k\le p(x)\\
0& \mbox{when }k>p(x)
\end{array}\right.
\quad \mbox{for }\nu\mbox{-a.e.\,}x,
\end{align*}
where $p(x)$ is the pointwise index of $(\cE,\cF)$.
\end{lemma}
\begin{proof}
The proof is based on the Gram--Schmidt orthogonalization.
Let 
\[
 \ell_\Q=\{u=(u_i)_{i\in\N}\in\ell_0\mid u_i\in\Q\mbox{ for all }i\in\N \},
\]
which is dense in $\ell_{Z(x)}$ for $\nu$-a.e.\,$x\in K$.
Here, by abuse of notation, an element of $\ell_0$ is also regarded as its equivalent class in $\ell_{Z(x)}$.
We regard each element of $\ell_\Q$ as an $\ell_0$-valued constant function on $K$ and $\ell_\Q$ as a subset of $\cD$.
Since $\ell_\Q$ is a countable set, we can write $\ell_\Q=\{u^{(1)},u^{(2)},u^{(3)},\ldots\}$ and $u^{(n)}=(u_i^{(n)})_{i\in\N}$.
We define a map 
\[
R\colon \cD\ni f=(f_i)_{i\in\N}\mapsto g=(g_i)_{i\in\N}\in\cD
\]
 by
\[
  g_i(x)=\left\{\begin{array}{cl}
  f_i(x)\left/\la (f_j(x))_{j\in\N},(f_j(x))_{j\in\N}\ra_{Z(x)}^{1/2}\right. & \mbox{if }(f_j(x))_{j\in\N}\not\sim_x 0\\
  0 & \mbox{otherwise}
  \end{array}\right.,\quad x\in K,\ i\in\N.
\]
We define $h^{(1)}=R(u^{(1)})$ and 
\[
h^{(n+1)}=R\left(u^{(n+1)}-\sum_{m=1}^n\la u^{(n+1)},h^{(m)}(\cdot)\ra_{Z(\cdot)}h^{(m)}\right),
\quad n\in\N,
\]
 inductively.
By taking account of \Lem{dimension}, for $\nu$-a.e.\,$x$, there exists a unique $n_1<n_2<\dots<n_{p(x)}$, $n_i\in\N$ for each $i$, such that
\[
  \la h^{(n_i)}(x),h^{(n_j)}(x)\ra_{Z(x)}=
  \begin{cases}1&(i=j)\\0&(i\ne j)\end{cases}
  \quad\mbox{for } i,j\in\{1,\dotsc,p(x)\}.
\]
Set 
\[
  g^{(k)}(x)=\left\{\begin{array}{cl}
  h^{(n_k)}(x) & \mbox{if }k\le p(x)\\
  0 & \mbox{if }k> p(x)
  \end{array}\right.
\]
for $k=1,\dotsc,p$. Here, note that $n_k$ depends on $x$.
Then, $g^{(1)},\dotsc,g^{(p)}$ satisfy the desired properties.
\end{proof}
\begin{proof}[of \Thm{index}]
First, we prove that the martingale dimension is less than or equal to the index $p$.
It suffices to assume $1\le p<\infty$.
Take $g^{(1)},\dotsc,g^{(p)}$ in \Lem{onb}, and define
\[
  K_l=\left\{x\in K\left|\begin{array}{l}  \mbox{For all $k=1,\dotsc,p$, $|g_i^{(k)}(x)|\le l$ for $i=1,\dots,l$}\\\mbox{and $g_i^{(k)}(x)=0$ for $i>l$}\end{array}\!\!\right.\right\}, \quad l\in\N.
\]
Then, $\{K_l\}_{l=1}^\infty$ is a nondecreasing sequence and $\nu\left(K\setminus\bigcup_{l=1}^\infty K_l\right)=0$.

Let $k\in\{1,\dotsc,p\}$ and $l\in\N$.
Keeping in mind that $(1_{K_l}\cdot g_i^{(k)})_{i\in\N}\in\cC$, we define
\[
M_l^{(k)}=\chi\left((1_{K_l}\cdot g_i^{(k)})_{i\in\N}\right)=\sum_{i=1}^\infty \left(1_{K_l}\cdot g_i^{(k)}\right)\bullet M^{[f_i]}.
\]
Note that $\sum_{i=1}^\infty$ in the above equation can be replaced by $\sum_{i=1}^l$.
For $l<m$, we have
\begin{align}\label{eq:eM}
e\left(M_m^{(k)}-M_l^{(k)}\right)
&=e\left(\sum_{i=1}^\infty \left(1_{K_m\setminus K_l}\cdot g_i^{(k)}\right)\bullet M^{[f_i]}\right)\nonumber\\
&=\frac12 \sum_{i,j=1}^\infty \int_{K_m\setminus K_l}g_i^{(k)}g_j^{(k)}\,d\mu_{\la M^{[f_i]},M^{[f_j]}\ra}\nonumber\\
&=\frac12\int_{K_m\setminus K_l}\left(\sum_{i,j=1}^\infty g_i^{(k)}g_j^{(k)}Z^{i,j}\right)\,d\nu
\qquad\mbox{(from \Eq{MAF})}
\nonumber\\
&=\frac12\int_{K_m\setminus K_l}1_{\{p(\cdot)\ge k\}}(x)\,\nu(dx) 
\qquad\mbox{(from \Lem{onb})}\nonumber\\
&\to 0 \qquad\mbox{as }m> l\to\infty.
\end{align}
Here, the infinite sums are actually finite sums.
From this, $\{M_l^{(k)}\}_{l=1}^\infty$ is a Cauchy sequence in $(\maruM,e)$. 
We denote the limit by $M^{(k)}$.
From a calculation similar to \Eq{eM}, we have 
$\mu_{\la M_l^{(k)}\ra}(dx)=1_{K_l\cap\{p(\cdot)\ge k\}}(x)\,\nu(dx)$ and $\mu_{\la M_l^{(k)},M_l^{(m)}\ra}=0$ for $k\ne m$, for every $l\in\N$.
By \Lem{M}~(ii), we obtain 
\begin{equation}\label{eq:muMk}
\mu_{\la M^{(k)}\ra}(dx)=1_{\{p(\cdot)\ge k\}}(x)\,\nu(dx)
\end{equation}
and
\begin{equation}\label{eq:muMkl}
\mu_{\la M^{(k)},M^{(m)}\ra}=0, \quad k\ne m
\end{equation}
by letting $l\to\infty$.

Now, we will prove that any $M\in\maruM$ is expressed as
\begin{equation}\label{eq:expression}
M=\sum_{k=1}^p h_k\bullet M^{(k)},\quad h_k=\frac{d\mu_{\la M,M^{(k)}\ra}}{d\nu},\ k=1,\dotsc,p.
\end{equation}
Here, since
\[
  h_k^2
  \le\frac{d\mu_{\la M\ra}}{d\nu}\cdot
  \frac{d\mu_{\la M^{(k)}\ra}}{d\nu}
  \le \frac{d\mu_{\la M\ra}}{d\nu}
  \quad \nu\mbox{-a.e.}
\]
from \Lem{M}~(i) and \Eq{muMk}, $h_k\in L^2(K,\nu)\subset L^2(K,\mu_{\la M^{(k)}\ra})$ and $h_k\bullet M^{(k)}$ is well-defined.

To begin with, suppose that $M$ is described as
\begin{equation}\label{eq:simple}
\begin{array}{l}
M=u\bullet M^{[f_i]},\mbox{ where $i\in\N$ and $u$ is a bounded Borel function on $K$}\\
\mbox{such that $u=0$ on $K\setminus K_l$ for some $l\in\N$}.
\end{array}
\end{equation}
We may assume $l\ge i$.
In order to prove \Eq{expression}, it suffices to prove that the Revuz measures of both sides coincide.
We have
\[
  \mu_{\la M\ra}(dx)=u(x)^2\mu_{\la f_i\ra}(dx)
  =u(x)^2Z^{i,i}(x)\nu(dx)
\]
and
\begin{align*}
  \mu_{\left\la \sum_{k=1}^p h_k\bullet M^{(k)}\right\ra}(dx)
  &=\sum_{k,m=1}^p h_k(x)h_m(x)\mu_{\la M^{(k)},M^{(m)}\ra}(dx)\\
  &=\sum_{k=1}^p h_k(x)^2 1_{\{p(\cdot)\ge k\}}(x)\nu(dx).
  \quad\mbox{(from \Eq{muMk} and \Eq{muMkl})}
\end{align*}
Since $\mu_{\la M^{[f_i]},M_l^{(k)}\ra}(dx)=1_{K_l}(x)\mu_{\la M^{[f_i]},M^{(k)}\ra}(dx)$ and $u(x)1_{K_l}(x)=u(x)$, we also have
\begin{align*}
h_k(x)&= u(x)\frac{d\mu_{\la M^{[f_i]},M^{(k)}\ra}}{d\nu}(x)\\
&= u(x)\frac{d\mu_{\la M^{[f_i]},M_l^{(k)}\ra}}{d\nu}(x)\\
&=\sum_{j=1}^l u(x)g_j^{(k)}(x)Z^{i,j}(x)\\
&= \la v(x),g^{(k)}(x)\ra_{Z(x)},
\end{align*}
where $v(x)=(v_j(x))_{j\in\N}\in\ell_0$ is defined as $v_j(x)=0$ for $j\ne i$ and $v_j(x)=u(x)$ for $j=i$.
For $\nu$-a.e.\,$x\in K$, $\{g^{(k)}(x)\}_{k=1}^{p(x)}$ is an orthonormal basis of $\ell_{Z(x)}$.
Therefore,
\[
  \sum_{k=1}^{p(x)}h_k(x)^2=\la v(x),v(x)\ra_{Z(x)}
  =u(x)^2 Z^{i,i}(x)
  \quad \nu\mbox{-a.e.\,}x.
\]
Combining these equalities, we conclude that $\mu_{\la M\ra}=\mu_{\left\la \sum_{k=1}^p h_k\bullet M^{(k)}\right\ra}$ and \Eq{expression} holds.

The set of all linear combinations of additive functionals $M$ expressed as \Eq{simple} is dense in $\maruM$, which is proved in the same manner as \Lem{density}. 
Thus, \Eq{expression} holds for every $M\in\maruM$ by approximation.
Therefore, the martingale dimension is less than or equal to $p$.

Next, we prove that the index is less than or equal to  the martingale dimension $p$. 
We may assume $1\le p<\infty$.
Then, there exist $M^{(1)},\dotsc,M^{(p)}\in\maruM$ such that every $M\in\maruM$ has a representation $M=\sum_{k=1}^p h_k\bullet M^{(k)}$.
Take arbitrary finitely many functions $g_1,\dotsc,g_N$ from $\cF$.
For each $i=1,\dotsc,N$, we write $M^{[g_i]}=\sum_{k=1}^p h_k^i\bullet M^{(k)}$.
Then,
\[
\nu_{g_i,g_j}
=\mu_{\la M^{[g_i]},M^{[g_j]}\ra}
=\sum_{m,n=1}^p h_m^i h_n^j\cdot\mu_{\la M^{(m)},M^{(n)}\ra}.
\]
Therefore, $Z_N(x)={}^tU(x)R(x)U(x)$, where
$Z_N(x)=\left((d\nu_{g_i,g_j}/{d\nu})(x)\right)_{i,j=1}^N$,
and $U(x)$ is a $p\times N$ matrix and $R(x)$ is a $p\times p$ matrix provided by
\[
U(x)=\left(h_k^i(x)\right)_{\begin{subarray}{l}k=1,\dotsc,p\\i=1,\dotsc,N\end{subarray}},\quad 
R(x)=\left(\frac{d\mu_{\la M^{(m)},M^{(n)}\ra}}{d\nu}(x)\right)_{m,n=1}^p.
\]
This implies that $\rank Z_N(x)\le p$ $\nu$-a.e.\,$x$. Thus, the index is less than or equal to $p$.
\end{proof}
\begin{remark}
The additive functionals $\{M^{(k)}\}_{k=1}^p$ constructed in the first part of the proof satisfy 
\[
  \mu_{\la M^{(k)}\ra}(dx)=1_{\{p(\cdot)\ge k\}}(x)\mu_{\la M^{(m)}\ra}(dx)\quad\mbox{for }k>m
\]
from \Eq{muMk}. Therefore, $\la M^{(k)}\ra_t=\int_0^t 1_{\{p(\cdot)\ge k\}}(X_s)\,d\la M^{(m)}\ra_s$ for $k>m$, where $p(\cdot)$ is taken to be Borel measurable, and in particular,
\[
  d\la M^{(1)}\ra_t \gg d\la M^{(2)}\ra_t \gg\cdots\gg d\la M^{(p)}\ra_t,
  \quad P_x\mbox{-a.e.\ for q.e.\,$x$}.
\]
This is consistent with the definition of the multiplicity of filtration in \cite{DV74}.
\end{remark}
\begin{corollary}\label{cor:one}
Suppose that the index of $(\cE,\cF)$ is one.
For $M'\in\maruM$, the following properties are equivalent.
\begin{enumerate}
\item $\mu_{\la M'\ra}$ is a minimal energy-dominant measure.
\item For any $M\in\maruM$, there exists $h\in L^2(K,\mu_{\la M'\ra})$ such that $M=h\bullet M'$.
\end{enumerate}
In particular, for $f\in\cF$, $M':=M^{[f]}$ satisfies property~(ii) above if and only if $f\in\cFz$.
\end{corollary}
\begin{proof}
From \Thm{index}, equation~\Eq{muMk}, and \Prop{zero}, there exists $M^{(1)}\in\maruM$ such that $\mu_{\la M^{(1)}\ra}=\nu$ and for each $M\in\maruM$, there exists $h_1\in L^2(K,\nu)$ satisfying $M=h_1\bullet M^{(1)}$.

Suppose (i) holds. There exists $\ph\in L^2(K,\nu)$ such that $M'=\ph\bullet M^{(1)}$.
Since $d\mu_{\la M'\ra}=\ph^2\,d\mu_{\la M^{(1)}\ra}= \ph^2\,d\nu$, $\ph\ne0$ $\nu$-a.e.
For $M\in\maruM$, take $h_1\in L^2(K,\nu)$ as above and define $h=h_1/\ph$. Then,
$
\int_K h^2\,d\mu_{\la M'\ra}=\int_K h_1^2\,d\nu<\infty$ and 
\[
h\bullet M'=(h_1/\ph) \bullet (\ph\bullet M^{(1)})= h_1 \bullet M^{(1)}=M.
\]
Therefore, (ii) holds.

Next, suppose (ii) holds. 
There exists $h\in L^2(K,\nu)$ such that $M^{(1)}=h\bullet M'$.
Then, $d\nu= d\mu_{\la M^{(1)}\ra}=h^2\,d\mu_{\la M'\ra}$.
This implies that $\nu\ll\mu_{\la M'\ra}$.
Therefore, (i) holds.
\end{proof}

\section{Index of p.c.f.\ fractals}
In this section, we take self-similar fractals as $K$.
We follow \cite{Ki} to provide a framework.
Let $K$ be a compact metrizable topological space, $N$ be an integer greater than one, and set $S=\{1,2,\ldots,N\}$.
Further, let $\psi_i\colon K\to K$ be a continuous injective map for $i\in S$.
Set $\Sg=S^\N$. For $i\in S$, we define a shift operator $\sg_i\colon \Sg\to\Sg$ by $\sg_i(\om_1\om_2\cdots)=i\om_1\om_2\cdots$.
Suppose that there exists a continuous surjective map $\pi\colon \Sg\to K$ such that $\psi_i\circ \pi=\pi\circ\sg_i$ for every $i\in S$.
We term $\cL=(K,S,\{\psi_i\}_{i\in S})$ a self-similar structure.

We also define $W_0=\{\emptyset\}$, $W_m=S^m$ for $m\in \N$, and
denote $\bigcup_{m\ge0}W_m$ by $W_*$.
For $w=w_1w_2\cdots w_m\in W_m$, we define $\psi_w=\psi_{w_1}\circ\psi_{w_2}\circ\cdots\circ\psi_{w_m}$
and $K_w=\psi_w(K)$.
By convention, $\psi_\emptyset$ is the identity map from $K$ to $K$.
For $w=w_1w_2\cdots w_m\in W_m$ and $w'=w'_1w'_2\cdots w'_{m'}\in W_{m'}$, $ww'$ denotes $w_1w_2\cdots w_mw'_1w'_2\cdots w'_{m'}\in W_{m+m'}$.
For $m\ge0$, let $\cB_m$ be a $\sg$-field on $K$ generated by $\{K_w\mid w\in W_m\}$.
Then, $\{\cB_m\}_{m=0}^\infty$ is a filtration on $K$ and the $\sg$-field generated by $\{\cB_m\mid m\ge0\}$ is equal to $\cB(K)$. 

We set 
\[
\cP=\bigcup_{m=1}^\infty \sg^m\left(\pi^{-1}\left(\bigcup_{i,j\in S,\,i\ne j}(K_i\cap K_j)\right)\right)\quad\text{and}\quad V_0=\pi(\cP),
\]
where $\sg^m\colon\Sigma\to\Sigma$ is a shift operator that is defined by $\sg^m(\om_1\om_2\cdots)=\om_{m+1}\om_{m+2}\cdots$.
The set $\cP$ is referred to as the post-critical set.
We assume that $K$ is connected and the self-similar structure $(K,S,\{\psi_i\}_{i\in S})$ is {\em post-critically finite} (p.c.f.), that is, $\cP$ is a finite set.
Figure~\ref{fig:pcf} shows some typical examples of p.c.f.\ self-similar sets $K$, where the set of black points is $V_0$.
The three-dimensional Sierpinski gasket is realized in $\R^3$, and other sets are realized in $\R^2$.
\begin{figure}[h]
\epsfig{file=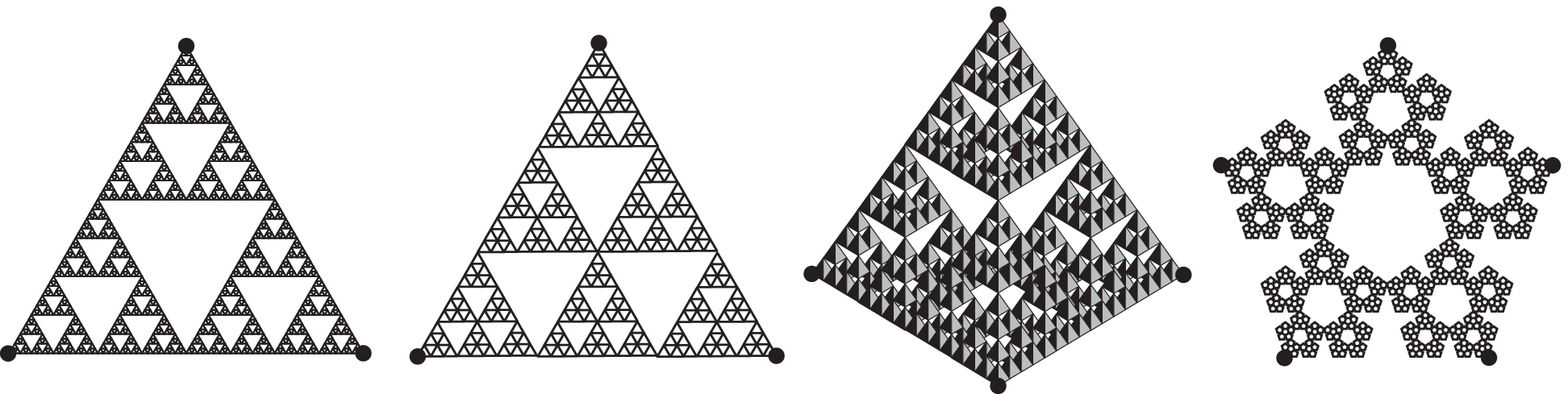, width=\textwidth}\medskip

\epsfig{file=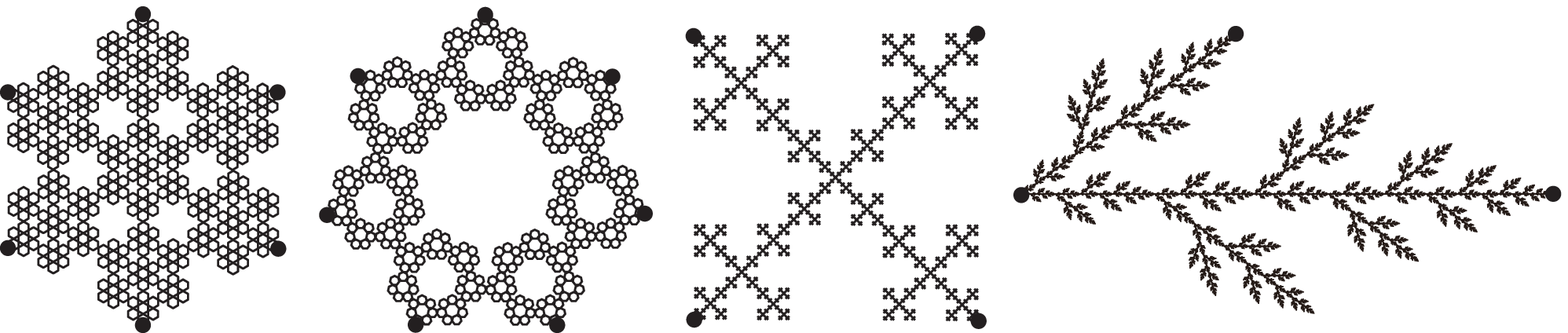, width=\textwidth}
\caption{Examples of p.c.f.\ self-similar sets. From the upper left, two-dimensional standard Sierpinski gasket, two-dimensional level-3 S. G., three-dimensional standard S. G., Pentakun (pentagasket), snowflake, heptagasket, the Vicsek set, and Hata's tree-like set.}
\label{fig:pcf}
\end{figure}

Let $V_m=\bigcup_{w\in W_m}\psi_w(V_0)$ for $m\in\N$ and $V_*=\bigcup_{m=0}^\infty V_m$.
For any $x\in K\setminus V_*$, there exists a unique element $\om=\om_1\om_2\cdots\in \Sg$ such that $\pi(\om)=x$.
We denote $\om_1\om_2\cdots\om_m\in W_m$ by $[x]_m$ for each $m\in\N$, and define $[x]_0=\emptyset$.
The sequence $\{K_{[x]_m}\}_{m=0}^\infty$ is a fundamental system of neighborhoods of $x$ (\cite[Proposition~1.3.6]{Ki}).

For a finite set $V$, let $l(V)$ be the space of all real-valued functions on $V$.
We equip $l(V)$ with an inner product $(\cdot,\cdot)_{l(V)}$   defined by $(u,v)_{l(V)}=\sum_{q\in V}u(q)v(q)$.
Let $D=(D_{qq'})_{q,q'\in V_0}$ be a symmetric linear operator on $l(V_0)$ (also considered to be a square matrix of size $\#V_0$) such that the following conditions hold:
\begin{enumerate}
\item[(D1)] $D$ is nonpositive definite,
\item[(D2)] $Du=0$ if and only if $u$ is constant on $V_0$,
\item[(D3)] $D_{qq'}\ge0$ for all $q\ne q'\in V_0$.
\end{enumerate}
We define $\cE^{(0)}(u,v)=(-Du,v)_{l(V_0)}$ for $u,v\in l(V_0)$.
This is a Dirichlet form on $l(V_0)$, where $l(V_0)$ is identified with the $L^2$ space on $V_0$ with the counting measure (\cite[Proposition~2.1.3]{Ki}).
For $\bfr=\{r_i\}_{i\in S}$ with $r_i>0$, we define a bilinear form $\cE^{(m)}$ on $l(V_m)$ as
\begin{equation}\label{eq:em}
  \cE^{(m)}(u,v)=\sum_{w\in W_m}\frac{1}{r_w}\cE^{(0)}(u\circ\psi_w|_{V_0},v\circ\psi_w|_{V_0}),\quad
  u,v\in l(V_m).
\end{equation}
Here, $r_w=r_{w_1}r_{w_2}\cdots r_{w_m}$ for $w=w_1w_2\cdots w_m$ and $r_\emptyset=1$.
We refer to $(D,\bfr)$ as a {\em harmonic structure} if for every $v\in l(V_0)$,
\[
\cE^{(0)}(v,v)=\inf\{\cE^{(1)}(u,u)\mid u\in l(V_1)\mbox{ and }u|_{V_0}=v\}.
\]
Then, for $m\in\Z_+$ and $u\in l(V_{m+1})$, we obtain $\cE^{(m)}(u|_{V_m},u|_{V_m})\le \cE^{(m+1)}(u,u)$.

We consider only a harmonic structure that is {\em regular}, namely, $0<r_i<1$ for all $i\in S$.
Several studies have been conducted on the existence of regular harmonic structures such as \cite{Li,HMT06,Pe07}. 
We only remark here that all nested fractals have canonical regular harmonic structures.
Nested fractals are self-similar sets that are realized in Euclidean spaces and have good symmetry. For the precise definition, see \cite{Li,Ki}. All the fractals shown in Figure~\ref{fig:pcf} except Hata's tree-like set are nested fractals.
Hata's tree-like set also has regular harmonic structures; see \Prop{Hata} below.

We assume that a regular harmonic structure $(D,\bfr)$ is given.
Let $\mu$ be a Borel probability measure on $K$ with full support.
We can then define a strong local and regular Dirichlet form $(\cE,\cF)$ on $L^2(K,\mu)$ associated with $(D,\bfr)$ by
\begin{align*}
\cF&=\left\{u\in C(K)\subset L^2(K,\mu)\left|\,
\lim_{m\to\infty}\cE^{(m)}(u|_{V_m},u|_{V_m})<\infty\right.\right\},\\
\cE(u,v)&= \lim_{m\to\infty}\cE^{(m)}(u|_{V_m},v|_{V_m}),\quad u,v\in\cF.
\end{align*}
(See the beginning of \cite[section~3.4]{Ki}.)

For a map $\psi\colon K\to K$ and a function $f$ on $K$, $\psi^* f$ denotes the pullback of $f$ by $\psi$, that is, $\psi^* f=f\circ \psi$.
The Dirichlet form $(\cE,\cF)$ constructed above satisfies the self-similarity:
\begin{align}\label{eq:selfsimilarity}
  \cE(f,g)=\sum_{i\in S}\frac1{r_i}\cE(\psi_i^* f,\psi_i^* g),\quad f,g\in\cF.
 \end{align}
From \cite[Theorem~3.3.4]{Ki}, there exists a constant $c_5>0$ such that
\begin{equation}\label{eq:poincare}
  \Osc_{x\in K}f(x)\le c_5\sqrt{\cE(f)},
  \quad f\in\cF\subset C(K),
\end{equation}
where, in general, $\Osc_{x\in B}f(x):=\sup_{x\in B}f(x)-\inf_{x\in B}f(x)$ for $B\subset K$.
By utilizing this inequality, it is easy to prove that the capacity associated with $(\cE,\cF)$ of any nonempty subset of $K$ is uniformly positive (see, e.g., \cite[Proposition~4.2]{Hi08}).

The energy measures have the following properties.
For the proof, see e.g.\ \cite[Lemma~4]{HN06} and its proof.
\begin{lemma}\label{lem:energymeas}
Let $f\in\cF$ and $\nu_f$ be the energy measure of $f$.
\begin{enumerate}
\item For $m\in\Z_+$ and a Borel subset $B$ of $K$,
\[
  \nu_f(B)=\sum_{w\in W_m}\frac1{r_w}\nu_{\psi_w^*f}(\psi_w^{-1}(B)).
\]

\item The measure $\nu_f$ has no atoms. 
In particular, $\nu_f(V_*)=0$.
\item For $w\in W_*$ and a Borel subset $B$ of $K_w$,
\[
  \nu_f(B)=\frac1{r_w}\nu_{\psi_w^*f}(\psi_w^{-1}(B)).
\]
\end{enumerate}
\end{lemma}
In addition, note that since $1\in\cF$ and $\cE(1,1)=0$, 
\begin{equation}\label{eq:total}
\nu_f(K)=2\cE(f) \quad\mbox{for all $f\in\cF$}.
\end{equation}
The underlying measure $\mu$ does not play an important role with regard to the energy measures since they are independent of the choice of $\mu$.

For each $u\in l(V_0)$, there exists a unique function $h\in\cF$ such that $h|_{V_0}=u$ and $\cE(h)=\inf\{\cE(g)\mid g\in\cF,\ g|_{V_0}=u\}$.
Such a function $h$ is termed a {\em harmonic function}.
The space of all harmonic functions is denoted by $\cH$.
For any $w\in W_*$ and $h\in\cH$, $\psi_w^* h\in\cH$.
By using the linear map $\iota\colon l(V_0)\ni u\mapsto h\in \cH$, we can identify $\cH$ with $l(V_0)$.
In particular, $\cH$ is a finite dimensional subspace of $\cF$, and the norm $\|\cdot\|_\cF$ on $\cH$ is equivalent to $(\iota^{-1}(\cdot), \iota^{-1}(\cdot))_{l(V_0)}^{1/2}$.
For each $i\in S$, we define a linear operator $A_i\colon l(V_0)\to l(V_0)$ as 
\begin{equation}\label{eq:A}
A_i=\iota^{-1}\circ \psi_i^* \circ \iota,
\end{equation}
which is also considered as a square matrix of size $\#V_0$.

Let $\cI$ be the set of all constant functions on $K$, which is a one-dimensional subspace of $\cH$.
Let $Q$ be an arbitrary projection operator from $\cH$ to $\cI$.
Define $P=\Id-Q$, where $\Id$ is an identity operator on $\cH$.
For each $i\in S$, set $A'_i=P A_i$.
For $w=w_1w_2\dotsm w_m\in W_m$, define $A_w=A_{w_m}\dotsm A_{w_2}A_{w_1}$ and $A'_w=A'_{w_m}\dotsm A'_{w_2}A'_{w_1}$.
By convention, $A_\emptyset=\Id$ and $A'_\emptyset=P$.
Let $\bfone\in l(V_0)$ be a constant function on $V_0$ with value $1$.
Since each $A_i$ has an eigenvalue $1$ with $\bfone$ as eigenfunctions, $A_iQ=Q$.
Then, $PA_iP=(\Id-Q)A_i(\Id-Q)=A_i-QA_i=PA_i$.
Therefore, for any $w\in W_*$,
\begin{equation}\label{eq:P}
A'_w=PA_w.
\end{equation}
In addition, since $DQ=0$ and $D$ is symmetric,
\begin{equation}\label{eq:D}
{}^tPDP=(\Id-{}^tQ)D(\Id-Q)=D. 
\end{equation}
For $m\ge0$, let $\cH_m$ denote the set of all functions $f$ in $\cF$ such that $\psi_w^*f\in\cH$ for all $w\in W_m$.
Let $\cH_*=\bigcup_{m\ge0}\cH_m$.
Functions in $\cH_*$ are referred to as {\em piecewise harmonic functions}.
From \cite[Lemma~2.1]{Hi08}, $\cH_*$ is dense in $\cF$. 
\begin{lemma}\label{lem:matrix}
Let $w\in W_*$ and $f,h\in\cH$.
Then, 
\begin{equation}\label{eq:nuh}
\nu_{f,h}(K_w)=-\frac2{r_w}{}^t(A'_w\iota^{-1}(f))D(A'_w\iota^{-1}(h)).
\end{equation}
\end{lemma}
\begin{proof}
It is sufficient to prove \Eq{nuh} when $f=h$.
From \cite[Lemma~4~(ii)]{HN06}, equation~\Eq{nuh} with  $A'_w$ replaced by $A_w$ holds.
By combining \Eq{P} and \Eq{D}, we obtain \Eq{nuh}.
\end{proof}

In what follows, we fix a minimal energy-dominant measure $\nu$ with $\nu(K)<\infty$.
\begin{proposition}\label{prop:equivalence}
Let $m\in\Z_+$.
Define $\nu'_m=\sum_{w\in W_m}r_w^{-1}(\psi_w)_*\nu$.
That is, 
\[
\nu'_m(B)=\sum_{w\in W_m}\frac{1}{r_w}\nu(\psi_w^{-1}(B)),
\quad B\in\cB(K).
\]
Then, $\nu$ and $\nu'_m$ are mutually absolutely continuous.
In particular, $\nu(B)=0$ implies $\nu(\psi_w^{-1}(B))=0$.
Moreover, for any $f,g\in \cF$ and $w\in W_m$,
\begin{equation}\label{eq:scale}
  \frac{d\nu_{f,g}}{d\nu'_m}(x)=\frac{d\nu_{\psi_w^* f,\psi_w^* g}}{d\nu}(\psi_w^{-1}(x))
  \quad \nu\mbox{-a.e.\,$x$ on }K_w.
\end{equation}
\end{proposition}
\begin{proof}
Let $B$ be a Borel set of $K$.
Suppose $\nu'_m(B)=0$.
Let $f\in\cF$ and $w\in W_m$. Then,
$0=((\psi_w)_*\nu)(B)=\nu(\psi_w^{-1}(B\cap K_w))$.
Since $\nu_{\psi_w^* f}\ll\nu$, we have
$0=\nu_{\psi_w^* f}(\psi_w^{-1}(B\cap K_w))=r_w\nu_f(B\cap K_w)$ by \Lem{energymeas}~(iii).
Since $w\in W_m$ is arbitrary, $\nu_f(B)=0$.
From \Lem{null}, $\nu(B)=0$.
Therefore, $\nu\ll\nu'_m$.

Next, suppose $\nu(B)=0$.
Let $f\in\cF$ and $w\in W_m$.
Then, there exists $\hat f\in\cF$ such that $\psi_w^* \hat f=f$.
From \Lem{energymeas}~(i), $0=\nu_{\hat f}(B)\ge r_w^{-1}\nu_{f}(\psi_w^{-1}(B))$.
From \Lem{null}, $\nu(\psi_w^{-1}(B))=0$, that is, $((\psi_w)_*\nu)(B)=0$.
Therefore, $\nu'_m(B)=0$.
This implies $\nu'_m\ll\nu$.

For the proof of \Eq{scale}, take $x\in K_w\setminus V_*$ and $n\ge m$.
Then, from \Lem{energymeas},
\[
  \frac{\nu_{f,g}(K_{[x]_n})}{\nu'_m(K_{[x]_n})}
  = \frac{r_w^{-1}\nu_{\psi_w^* f, \psi_w^* g}(\psi_w^{-1}(K_{[x]_n}))}
  {r_w^{-1}\nu(\psi_w^{-1}(K_{[x]_n}))}
  = \frac{\nu_{\psi_w^* f, \psi_w^* g}(K_{[\psi_w^{-1}(x)]_{n-m}})}
  {\nu(K_{[\psi_w^{-1}(x)]_{n-m}})}.
\]
Letting $n\to\infty$, we obtain \Eq{scale}.
\end{proof}
For each $q\in V_0$, define $h_q\in\cH$ so that $h_q(x)=\begin{cases}1&(x=q)\\0&(x\ne q)\end{cases}$, $x\in V_0$. 
Let $Z^{q,q'}=\left.d\nu_{h_q,h_{q'}}\right/d\nu$ for $q,q'\in V_0$.
\begin{proposition}\label{prop:indexK}
Define 
$
p=\nu\mbox{-}\!\esssup_{x\in K} \rank\left(Z^{q,q'}(x)\right)_{q,q'\in V_0}
$.
Then, the index of $(\cE,\cF)$ is $p$.
\end{proposition}
\begin{proof}
It suffices to prove that the index is less than or equal to $p$.
Take a sequence $\{f_i\}_{i=1}^\infty$ from $\cH_*$ such that the linear span of $\{f_i\}_{i=1}^\infty$ is dense in $\cF$.
For $N\in\N$, define an $N\times N$ matrix $Z_N(x)$, $x\in K$, by
$
  Z_N(x)=\left(({d\nu_{f_i,f_j}}/{d\nu})(x)\right)_{i,j=1}^N.
$
From \Prop{rank}, it is sufficient to show that 
\begin{equation}\label{eq:esssup}
\nu\mbox{-}\!\esssup_{x\in K} \rank Z_N(x)\le p.
\end{equation}
Take a sufficiently large $m\in\N$ such that $f_i\in \cH_m$ for all $i=1,\dotsc, N$.
Let $w\in W_m$.
Since $\psi_w^* f_i$ is expressed as a linear combination of $\{h_q\}_{q\in V_0}$ for every $i=1,\dotsc, N$, it is easy to prove 
\[
\nu\mbox{-}\!\esssup_{x\in K} \rank \left(\frac{d\nu_{\psi_m^*f_i, \psi_m^*f_j}}{d\nu}(x)\right)_{i,j=1}^N\le p,
\]
as in the first part of the proof of \Prop{rank}.
From \Prop{equivalence},
\[
  \frac{d\nu_{f_i,f_j}}{d\nu}(x)=\frac{d\nu'_m}{d\nu}(x) \frac{d\nu_{\psi_w^* f_i,\psi_w^* f_j}}{d\nu}(\psi_w^{-1}(x))
  \quad \nu\mbox{-a.e.\ on }K_w,
\]
where $\nu'_m$ is provided in \Prop{equivalence}.
Therefore, $\nu\mbox{-}\!\esssup_{x\in K_w} \rank Z_N(x)\le p$.
Since $w\in W_m$ is arbitrary, we obtain \Eq{esssup}. 
\end{proof}
For determining the index of $(\cE,\cF)$, the following result is the most general one available at present, which is an extension of the result described in \cite{Ku89}.
\begin{theorem}[({\cite[Theorem~4.4]{Hi08}})]\label{th:main0}
Suppose that for each $q\in V_0$, $q$ is a fixed point of $\psi_i$ for some $i\in S$ and $K\setminus \{q\}$ is connected.
Then, the index (or equivalently, the martingale dimension) of $(\cE,\cF)$ is one.
\end{theorem}
The author expects that the index is always one without the assumption of \Thm{main0}, but does not know the proof nor the counterexamples at the moment.
All nested fractals satisfy the assumption of \Thm{main0} (see e.g.\ \cite[Theorem~1.6.2 and Proposition~1.6.9]{Ki} for the proof).
A typical example that does not satisfy the assumption is Hata's tree-like set (shown in the lower right-hand side of Figure~\ref{fig:pcf}).
More precisely, let $c\in\C\setminus\R$ such that $0<|c|<1$ and $0<|1-c|<1$.
Hata's tree-like set is a unique nonempty compact subset $K$ of $\C$ such that $K=\psi_1(K)\cup \psi_2(K)$,
where $\psi_1(z)=c\bar z$ and $\psi_2(z)=(1-|c|^2)\bar z+|c|^2$ for $z\in\C$.
Then, $V_0=\{c,0,1\}$ and $c$ is not a fixed point of $\psi_1$ nor $\psi_2$.
However, even in this case, we can prove that the index is one by the following direct computation.\begin{figure}[h]
\epsfig{file=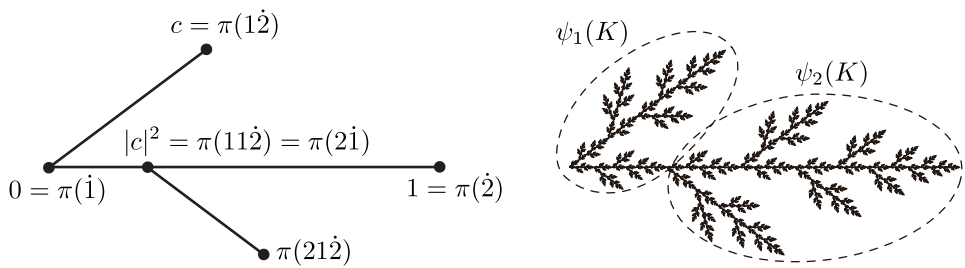}
\caption{Skeleton of Hata's tree-like set and images by contraction maps.}
\label{fig:hata}
\end{figure}
\begin{proposition}\label{prop:Hata}
Let $K$ be Hata's tree-like set and $\mu$, a Borel probability measure on $K$ with full support.
For a Dirichlet form $(\cE,\cF)$ on $L^2(K,\mu)$ associated with an arbitrary regular harmonic structure $(D,\bfr)$, the index is one.
\end{proposition}
\begin{proof}
From an elementary calculation, all harmonic structures $(D,\bfr)$ are provided by
\[
  D=a\begin{pmatrix} -h & h& 0\\ h& -(h+1) &1 \\ 0 & 1 & -1\end{pmatrix},\quad
  \bfr=(r,1-r^2)
\]
for $a>0$, $0<r<1$ and $rh=1$ (cf.~\cite[Example~3.1.6]{Ki}), where we identify $l(\{c,0,1\})$ with $\R^3$.
We may assume $a=1$ to prove the claim.
Then, the matrices $A_1$ and $A_2$ defined by \Eq{A} are given by
\[
  A_1=\begin{pmatrix} 0 & 1-r^2 & r^2 \\ 0&1&0\\ 1&0&0\end{pmatrix},\quad
  A_2=\begin{pmatrix} 0 & 1-r^2 & r^2 \\ 0&1-r^2&r^2\\ 0&0&1\end{pmatrix}.
\]
As projections on $\cH$, take $Q=\begin{pmatrix}0&0&1\\0&0&1\\0&0&1\end{pmatrix}$ and $P=\Id-Q=\begin{pmatrix}1&0&-1\\0&1&-1\\0&0&0\end{pmatrix}$.
Then,
\[
  A'_1:=PA_1=\begin{pmatrix} -1 & 1-r^2 & r^2 \\ -1&1&0\\ 0&0&0\end{pmatrix}\quad\mbox{and}\quad
  A'_2:=PA_2=\begin{pmatrix} 0 & 1-r^2 & r^2-1 \\ 0&1-r^2&r^2-1\\ 0&0&0\end{pmatrix}.
\]
Therefore, $\rank A'_1=2$ and $\rank A'_2=1$.
This implies that for $w\in W_m$ with $m\in\N$, $\rank A'_w\le1$ unless $w=\underbrace{11\dotsm1}_{m}$.

For $q,q'\in V_0$, define $h_q$ and $Z^{q,q'}$ as in the statement before \Prop{indexK}.
For $n\in\N$ and $x\in K\setminus V_*$, let $Z_n^{q,q'}(x)=\nu_{h_q,h_{q'}}(K_{[x]_n})/\nu(K_{[x]_n})$, where $0/0:=1$.
Then, from \Lem{matrix}, matrix $\left(Z_n^{q,q'}(x)\right)_{q,q'\in V_0}$ is equal to $-2(r_{[x]_n}\nu(K_{[x]_n}))^{-1}\,{}^t\!A'_{[x]_n} D A'_{[x]_n}$.
Its rank is less than or equal to $1$ unless $[x]_n= \underbrace{11\dotsm1}_{n} $.
Since $\left(Z_n^{q,q'}(x)\right)_{q,q'\in V_0}$ converges to $Z(x):=\left(Z^{q,q'}(x)\right)_{q,q'\in V_0}$ $\nu$-a.e.\,$x$ as $n\to\infty$ by the martingale convergence theorem, and $\bigcap_{n=1}^\infty \{x\in K\setminus V_*\mid [x]_n= \underbrace{11\dotsm1}_{n}\}=\emptyset$, we conclude that $\rank Z(x)\le 1$ $\nu$-a.e.\,$x$.
From \Prop{indexK} and \Prop{zero}, the index is one.
\end{proof}
\begin{remark}
The index discussed in this section may be different from that defined in~\cite{PT08}.
The author does not know whether these two definitions are equivalent or not.
\end{remark}
\section{Derivatives of functions on p.c.f.\ fractals}
In this section, we discuss the concept of derivatives of functions on p.c.f.\ self-similar fractals.
We use the same notations as those in the previous section.
\begin{lemma}\label{lem:upper}
There exists a constant $c_6>0$ such that for any $f\in\cF$, $x\in K\setminus V_*$, and $n\in\Z_+$,
\[
  \Osc_{y\in K_{[x]_n}} f(y)\le c_6 \sqrt{r_{[x]_n} \nu_f(K_{[x]_n})}.
\]
\end{lemma}
\begin{proof*}
From \Eq{poincare}, \Eq{total}, and \Lem{energymeas}~(iii), we have
\[ \singlebox
\Osc_{y\in K_{[x]_n}} f(y)
=\Osc_{y\in K}(\psi^*_{[x]_n}f)(y)
\le c_5\sqrt{\cE(\psi^*_{[x]_n}f)}
=c_5\sqrt{\frac{1}{2}r_{[x]_n}\nu_f(K_{[x]_n})}.
\esinglebox \]
\end{proof*}

For $n\in\Z_+$, define a map $H_n\colon \cF\to\cH_n$ so that $H_n(f)$ is a unique function in $\cH_n$ satisfying $f=H_n(f)$ on $V_n$.
Note that $\cE(H_n(f))\le \cE(f)$ for all $n$ and $\|f-H_n(f)\|_\cF\to0$ as $n\to\infty$.
We write $H$ for $H_0$.

Let $g\in\cFz$, where $\cFz$ is defined in \Eq{Fzero}.
Let $g_n$ denote $H_n(g)$ for $n\in\Z_+$.
Take a sequence $\{n(k)\}\uparrow\infty$ such that
\begin{equation}\label{eq:convergent}
\sum_{k=1}^\infty \sqrt{\cE(g-g_{n(k)})}<\infty.
\end{equation}
\begin{lemma}\label{lem:lower1}
For $\nu_g$-a.e.\,$x$, 
\[
  \lim_{k\to\infty}\frac{\nu_{g_{n(k)}}(K_{[x]_{n(k)}})}{\nu_g(K_{[x]_{n(k)}})}=1.
\]
Here, by convention, $0/0:=1$.
\end{lemma}
\begin{proof}
For each $k\in\N$, from \Eq{energyineq}, 
\begin{align*}
\int_K\left|1-\frac{d\nu_{g_{n(k)}}}{d\nu_g}\right|\,d\nu_g
&\le\sqrt{2\cE(g-g_{n(k)})}\left(\sqrt{2\cE(g)}+\sqrt{2\cE(g_{n(k)})}\right)\\
&\le 4\sqrt{\cE(g)}\sqrt{\cE(g-g_{n(k)})}.
\end{align*}
Let $M_k$ be the conditional expectation of $\left|1-\frac{d\nu_{g_{n(k)}}}{d\nu_g}\right|$ with respect to $\nu_g$ for given $\cB_{n(k)}$, that is,
\[
  M_k(x)=\frac1{\nu_g(K_{[x]_{n(k)}})}\int_{K_{[x]_{n(k)}}}\left|1-\frac{d\nu_{g_{n(k)}}}{d\nu_g}\right|\,d\nu_g,\quad x\in K\setminus V_*,
\]
where the right-hand side is defined as $0$ if $\nu_g(K_{[x]_{n(k)}})=0$.
Then, 
\[
\int_K M_k\,d\nu_g 
= \int_K\left|1-\frac{d\nu_{g_{n(k)}}}{d\nu_g}\right|\,d\nu_g
\le 4\sqrt{\cE(g)}\sqrt{\cE(g-g_{n(k)})}.
\]
From \Eq{convergent}, 
\[
\int_K \left(\sum_{k=1}^\infty M_k\right)\,d\nu_g=\sum_{k=1}^\infty\int_K M_k\,d\nu_g<\infty.
\]
Therefore, for $\nu_g$-a.e.\,$x$, $\sum_{k=1}^\infty M_k(x)<\infty$, in particular, $\lim_{k\to\infty}M_k(x)=0$.
Since
\[
M_k(x)
\ge \frac1{\nu_g(K_{[x]_{n(k)}})}\left|\int_{K_{[x]_{n(k)}}}\left(1-\frac{d\nu_{g_{n(k)}}}{d\nu_g}\right)\,d\nu_g\right|
=\left|1-\frac{\nu_{g_{n(k)}}(K_{[x]_{n(k)}})}{\nu_{g}(K_{[x]_{n(k)}})}\right|,
\]
we obtain the claim.
\end{proof}
\begin{lemma}\label{lem:lower2}
There exists a constant $c_7>0$ satisfying the following:
for $\nu_g$-a.e.\,$x$, there exists $k_0(x)\in\N$ such that
\[
  \Osc_{y\in K_{[x]_{n(k)}}}g(y)
  \ge c_7 \sqrt{r_{[x]_{n(k)}}\nu_g(K_{[x]_{n(k)}})},
  \quad k\ge k_0(x).
\]
Moreover, for $\nu_g$-a.e.\,$x$, $\Osc_{y\in K_{[x]_{n}}}g(y)>0$ for every $n\in\Z_+$.
\end{lemma}
\begin{proof}
From \Lem{lower1}, for $\nu_g$-a.e.\,$x$, there exists $k_0(x)\in\N$ such that 
\begin{equation}\label{eq:lower21}
\nu_{g_{n(k)}}(K_{[x]_{n(k)}})\ge\frac12\nu_g(K_{[x]_{n(k)}}),
\quad k\ge k_0(x).
\end{equation}
Let $\widehat\cH=\{h\in\cH\mid \int_K h\,d\mu=0\}$.
Since $\widehat\cH$ is a finite dimensional space and both of the maps $\widehat\cH\ni h\mapsto \Osc_{y\in V_0}h(y)\in \R$ and $\widehat\cH\ni h\mapsto \sqrt{\cE(h)}\in\R$ provide norms on $\widehat\cH$, there exists a constant $c_8>0$ depending only on $(\cE,\cF)$ such that for every $h\in\widehat\cH$, 
\[
  \sqrt{\cE(h)}\le c_8 \Osc_{y\in V_0}h(y).
\]
This inequality also holds  for $h\in\cH$.
Since $\psi_{[x]_{n(k)}}^* g_{n(k)}$ coincides with $H(\psi_{[x]_{n(k)}}^* g)\in\cH$, we have
\begin{align}\label{eq:lower22}
\sqrt{r_{[x]_{n(k)}}\nu_{g_{n(k)}}(K_{[x]_{n(k)}})}
&=\sqrt{\nu_{H(\psi_{[x]_{n(k)}}^* g)}(K)} \qquad\mbox{(from \Lem{energymeas}~(iii))}\nonumber\\
&= \sqrt{2\cE({H(\psi_{[x]_{n(k)}}^* g)})} \qquad \mbox{(from \Eq{total})}\nonumber\\
&\le \sqrt{2}c_8\Osc_{y\in V_0}H(\psi_{[x]_{n(k)}}^* g)(y)\nonumber\\
&=  \sqrt{2}c_8\Osc_{y\in \psi_{[x]_{n(k)}} (V_0)}g(y)\nonumber\\
&\le  \sqrt{2}c_8\Osc_{y\in K_{[x]_{n(k)}}}g(y).
\end{align}
Combining \Eq{lower21} and \Eq{lower22}, we obtain the first claim.

To prove the second claim, let 
$B=\{x\in K\setminus V_*\mid\nu_g(K_{[x]_{n}})=0 \mbox{ for some }n\in\Z_+\}$.
 Since $\{K_{[x]_{n}}\}_{n=0}^\infty$ is fundamental system of neighborhoods of each $x\in K\setminus V_*$, $B$ is nothing but $K\setminus( V_*\cup\supp\nu_g)$.
Therefore, $\nu_g(B)=0$.
This and the first claim imply the second claim.
\end{proof}
The following is one of the main theorems of this section.
\begin{theorem}\label{th:main}
Suppose that the index of $(\cE,\cF)$ is one.
Let $g\in\cFz$.
Then, for every $f\in\cF$, for $\nu_g$-a.e.\,$x$, there exists a unique real number $\frac{df}{dg}(x)$ such that
\begin{equation}\label{eq:R}
  f(y)-f(x)=\frac{df}{dg}(x)(g(y)-g(x))+R_x(y),\quad y\in K,
\end{equation}
where $R_x(y)$ is negligible in the following sense:\begin{equation}\label{eq:negligible}
  \lim_{k\to\infty}\frac{\sup_{y\in K_{[x]_{n(k)}}}|R_x(y)|}{\Osc_{y\in K_{[x]_{n(k)}}} g(y)}=0,
\end{equation}
where the sequence $\{n(k)\}\uparrow\infty$ is taken so that \Eq{convergent} holds.
In particular, 
\begin{equation}\label{eq:negligible2}
  \liminf_{n\to\infty}\frac{\sup_{y\in K_{[x]_{n}}}|R_x(y)|}{\Osc_{y\in K_{[x]_{n}}} g(y)}=0.
\end{equation}
Moreover, it holds that
\begin{equation}\label{eq:Ef}
\cE(f)=\frac12\int_K\left(\frac{df}{dg}(x)\right)^2\,\nu_g(dx),
\end{equation}
and the map $\cF\ni f\mapsto \frac{df}{dg}\in L^2(K,\nu_g)$ is a bounded linear map.
If $g$ belongs to $\cH_*$, \Eq{negligible} can be replaced by the stronger claim
\[
\lim_{n\to\infty}\frac{\sup_{y\in K_{[x]_n}}|R_x(y)|}{\Osc_{y\in K_{[x]_n}} g(y)}=0.
\]
\end{theorem}
\begin{proof}
Take an arbitrary c.o.n.s.\ $\{f_i\}_{i=1}^\infty$ of $\cF$, and define $Z^{i,j}=d\nu_{f_i,f_j}/d\nu_g$ for $i,j\in\N$.
From \Prop{squareroot} and the assumption that the index is one, there exists a sequence $\{\zt^i\}_{i=1}^\infty$ of real-valued functions on $K$ such that $Z^{i,j}(x)=\zt^i(x)\zt^j(x)$ $\nu_g$-a.e.\,$x$.
For $f\in\cF$, $f$ is uniquely expressed as $\sum_{i=1}^\infty a_i f_i$ for some $(a_i)_{i\in\N}\in\ell_2$.
Define a map $\gm\colon \cF\to L^2(K,\nu_g)$ by $\gm(f)=\sum_{i=1}^\infty a_i \zt^i$ for $f=\sum_{i=1}^\infty a_i f_i$.
This is well-defined; indeed, for $f=\sum_{i=1}^k a_i f_i$ and $h=\sum_{i=1}^k b_if_i$ for some $k\in\N$,
\[
  \left(\sum_{i=1}^k a_i\zt^i(x)\right)\left(\sum_{j=1}^k b_j\zt^j(x)\right)
  =\sum_{i,j=1}^k a_ib_j Z^{i,j}(x)
  =\frac{d\nu_{f,h}}{d\nu_g}(x)
  \quad \mbox{for }\nu_g\mbox{-a.e.}\,x
\]
and
\[
  \int_K \frac{d\nu_{f,h}}{d\nu_g}(x)\,\nu_g(dx)
  =\nu_{f,h}(K)=2\cE(f,h),
\]
which imply that $\gm$ is not only well-defined as a bounded linear operator from $\cF$ to $L^2(K,\nu_g)$ but it also satisfies the relation
\begin{equation}\label{eq:gamma}
  \gm(f)\gm(h)=\frac{d\nu_{f,h}}{d\nu_g}\quad
  \nu_g\mbox{-a.e.}
\end{equation}
for every $f,h\in\cF$.
In particular, $|\gm(g)|=1$ $\nu_g$-a.e.
We write $\gm_f$ for $\gm(f)$.

Take $f\in\cF$ and define $a(x)=\gm_f(x)\gm_g(x)$.
We will show that we can take $a(x)$ as $\frac{df}{dg}(x)$.
For $x\in K$, define
\[
  R_x(y)=f(y)-f(x)-a(x)(g(y)-g(x)),\quad y\in K.
\]
Then, $R_x(\cdot)$ belongs to $\cF$, and we have
\begin{align}\label{eq:negl}
\frac{\nu_{R_x}(K_{[x]_n})}{\nu_g(K_{[x]_n})}
&=\frac{\nu_{f(\cdot)-a(x)g(\cdot)}(K_{[x]_n})}{\nu_g(K_{[x]_n})}\nonumber\\
&=\frac{\nu_f(K_{[x]_n})-2a(x)\nu_{f,g}(K_{[x]_n})+a(x)^2\nu_g(K_{[x]_n})}{\nu_g(K_{[x]_n})}\nonumber\\
&\stackrel{n\to\infty}{\longto} \frac{d\nu_f}{d\nu_g}(x)-2a(x)\frac{d\nu_{f,g}}{d\nu_g}(x)+a(x)^2\nonumber\\
&\qquad\quad \mbox{(for $\nu_g$-a.e.\,$x$, from the martingale convergence theorem)}\nonumber\\
&=\gm_f(x)^2-2a(x)\gm_f(x)\gm_g(x)+a(x)^2\qquad\mbox{(from \Eq{gamma})}\nonumber\\
&=0 \qquad \mbox{for }\nu_g\mbox{-a.e.\,}x.
\end{align}
Take a sequence $\{n(k)\}\uparrow\infty$ such that \Eq{convergent} holds.
From $R_x(x)=0$, \Lem{upper} and \Lem{lower2}, for $\nu_g$-a.e.\,$x$, for sufficiently large $k\in \N$,
\begin{equation}\label{eq:comparison}
\frac{\sup_{y\in K_{[x]_{n(k)}}}|R_x(y)|}{\Osc_{y\in K_{[x]_{n(k)}}} g(y)}
\le\frac{\Osc_{y\in K_{[x]_{n(k)}}}R_x(y)}{\Osc_{y\in K_{[x]_{n(k)}}} g(y)}
\le  \frac{c_6}{c_7}\sqrt{\frac{\nu_{R_x}(K_{[x]_{n(k)}})}{\nu_g(K_{[x]_{n(k)}})}}.
\end{equation}
Combining \Eq{negl} and \Eq{comparison}, we obtain \Eq{negligible}.
Since
\[
  \int_K\left(\frac{df}{dg}(x)\right)^2\,\nu_g(dx)
  = \int_K\gm_f(x)^2\,\nu_g(dx)
  =\int_K \frac{d\nu_f}{d\nu_g}\,d\nu_g
  =2\cE(f),
\]
we obtain \Eq{Ef}.

Next, we prove the uniqueness of $\frac{df}{dg}(x)$.
In fact, the weaker property \Eq{negligible2} does imply the uniqueness.
Suppose two functions $\widehat{\frac{df}{dg}}(x)$, $\widetilde{\frac{df}{dg}}(x)$ and remainder terms $\widehat R_x(y)$, $\widetilde R_x(y)$ satisfy \Eq{R} and \Eq{negligible2} with $\frac{df}{dg}(x)$ and $R_x(y)$ replaced by $\widehat{\frac{df}{dg}}(x)$, $\widetilde{\frac{df}{dg}}(x)$ and $\widehat R_x(y)$, $\widetilde R_x(y)$, respectively.
We can take an increasing sequence $\{n'(k)\}\uparrow\infty$ such that \Eq{convergent} holds with $\{n(k)\}$ replaced by $\{n'(k)\}$,
\[
  \lim_{k\to\infty}\frac{\sup_{y\in K_{[x]_{n'(2k)}}}|\widehat R_x(y)|}{\Osc_{y\in K_{[x]_{n'(2k)}}} g(y)}=0,
  \quad\mbox{and}\quad
   \lim_{k\to\infty}\frac{\sup_{y\in K_{[x]_{n'(2k+1)}}}|\widetilde R_x(y)|}{\Osc_{y\in K_{[x]_{n'(2k+1)}}} g(y)}=0.
\]
There exists $\frac{df}{dg}(x)$ for $\nu_g$-a.e.\,$x$ such that \Eq{R} and \Eq{negligible} hold with $n(k)$ replaced by $n'(k)$ in \Eq{negligible}.
We have
\begin{equation}\label{eq:difference}
  R_x(y)-\widehat R_x(y)
  =\left(\widehat{\frac{df}{dg}}(x)-\frac{df}{dg}(x)\right)(g(y)-g(x)),
  \quad y\in K.
\end{equation}
From \Lem{lower2}, for $\nu_g$-a.e.\,$x$, $\Osc_{y\in K_{[x]_{n}}}g(y)>0$ for every $n$.
Take such $x$ from $K\setminus V_*$.
For each $n\in\Z_+$, choose $y_n\in K_{[x]_n}$ such that $|g(y_n)-g(x)|\ge(1/2)\Osc_{y\in K_{[x]_n}}g(y)$.
From \Eq{difference}, for every $k\in\N$,
\begin{align*}
  \left| \widehat{\frac{df}{dg}}(x)-\frac{df}{dg}(x)\right|
  &=\left|\frac{R_x(y_{n'(2k)})-\widehat R_x(y_{n'(2k)})}{g(y_{n'(2k)})-g(x)}\right|\\
  &\le \frac{2\sup_{y\in K_{[x]_{n'(2k)}}}|R_x(y)|}{\Osc_{y\in K_{[x]_{n'(2k)}}} g(y)}
  + \frac{2\sup_{y\in K_{[x]_{n'(2k)}}}|\widehat R_x(y)|}{\Osc_{y\in K_{[x]_{n'(2k)}}} g(y)}.
\end{align*}
By letting $k\to\infty$, we have $\widehat{\frac{df}{dg}}(x)=\frac{df}{dg}(x)$ for $\nu_g$-a.e.\,$x$.
In the same way, we can prove that $\widetilde{\frac{df}{dg}}(x)=\frac{df}{dg}(x)$ for $\nu_g$-a.e.\,$x$.
Therefore, the uniqueness holds.

The last claim is obvious, since if $g\in \cH_m$ for some $m\in\Z_+$, then $H_n(g)=g$ for every $n\ge m$ and \Eq{convergent} holds by letting $n(k)=k$, $k\in\N$.
\end{proof}
\begin{remark}
\begin{enumerate}
\item 
  Pelander and Teplyaev obtained a result similar to  \Thm{main}~(\cite[Theorem~4.8]{PT08}); however, they proved the a.e.\ existence of derivatives  with respect to some self-similar measures.
  Since self-similar measures and energy measures are mutually singular in many cases (\cite{Hi05,HN06}), their result and ours are not comparable.

\item
  Even when $K$ is not a self-similar fractal, \Thm{main} should hold with suitable modification.
  Moreover, when the index $p$ is greater than one, it is expected that there exist suitable reference functions $g_1,\dotsc,g_p$ such that any function $f$ in $\cF$ has a Taylor expansion with respect to $g_1,\dotsc,g_p$ up to the first order.
  We leave such generalization for future studies.
\end{enumerate}
\end{remark}

In \Thm{main}, the reference function $g$ can be taken from a dense subset $\cFz$ of $\cF$; however, exactly what type of function can be chosen remains unresolved.
We will provide a partial solution to this problem.
Take $d$-dimensional level-$l$ Sierpinski gaskets as $K$ for $d\ge2$ and $l\ge2$ (see Figure~\ref{fig:sg}).
\begin{figure}[h]
\epsfig{file=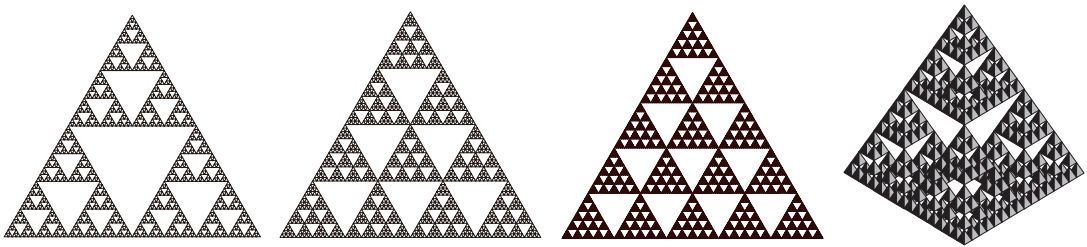, width=\textwidth}
\caption{two-dimensional level-$l$ Sierpinski gaskets $(l=2,4,5)$\newline and three-dimensional level-2 Sierpinski gasket.}
\label{fig:sg}
\end{figure}
Set $V_0$ consists of $(d+1)$ points denoted as $\{q_1,\dotsc,q_{d+1}\}$. 
We may assume that the index set $S$ of contraction maps $\{\psi_i\}_{i\in S}$ is a subset of $\N$, and for each $i\in\{1,\dotsc,d+1\}$, $q_i$ is a fixed point of $\psi_i$. 
There exists a canonical harmonic structure $(D,\bfr)$ since $K$ is a nested fractal, where $D=(D_{qq'})_{q,q'\in V_0}$ is provided by
\[
  D_{qq'}=\left\{\!\!\begin{array}{cl}-cd&\text{if }q=q'\\
  c& \text{if }q\ne q'\end{array}\right.
\]
for some $c>0$ and
$\bfr=(r_i)_{i\in S}$ satisfies $r_1=r_2=\dots=r_{d+1}=r$ for some $r\in(0,1)$.
We may assume $c=1$ without loss of generality.
Let $\mu$ be the normalized Hausdorff measure on $K$ and $(\cE,\cF)$, the Dirichlet form on $L^2(K,\mu)$ associated with $(D,\bfr)$.
We further assume that $(\cE,\cF)$ (or $(D,\bfr)$) is {\em nondegenerate}, namely, for every $i\in S$, the map $A_i\colon l(V_0)\to l(V_0)$ defined in \Eq{A} is bijective. In other words, there are no nonconstant harmonic functions that are constant on some small cells.
The author expects that this assumption is always true, although it has not yet been proved in general.
The following cases are nondegenerate:
\begin{enumerate}
\item $d\ge2$, $l=2$ (standard $d$-dimensional Sierpinski gaskets).
\item $d=2$, $2\le l\le50$.
\item $d=3$, $l=3$.
\end{enumerate}
In cases (1) and (3), the matrices $A_i$ are explicitly calculated and are proved to be invertible. In case~(2), the author checked the nondegeneracy by numerical computation with the aid of computers.
\begin{theorem}\label{th:sg}
Under the nondegeneracy condition, every nonconstant harmonic function $h$ belongs to $\cFz$, that is, its energy measure $\nu_h$ is a minimal energy-dominant measure.
In particular, for any nonconstant harmonic functions $h_1$ and $h_2$, $\nu_{h_1}$ and $\nu_{h_2}$ are mutually absolutely continuous.
\end{theorem}
For the proof, we prepare a series of lemmas.
For $i\in\{1,\dotsc,d+1\}$, define $h_i\in\cH$ so that $h_i(x)=\begin{cases}1&(x=q_i)\\0&(x\ne q_i)\end{cases}$ for $x\in V_0$.
Let $\nu=\sum_{i=1}^{d+1}\nu_{h_i}$.
\begin{lemma}\label{lem:sum}
Measure $\nu$ is a minimal energy-dominant measure.
\end{lemma}
\begin{proof}
Since $A_i\colon l(V_0)\to l(V_0)$ is bijective for every $i\in S$, $\psi_w^*\colon \cH\to \cH$ is also bijective for every $w\in W_*$.

Let $m\in\Z_+$ and $f\in\cH_m$.
For any $w\in W_m$, $\psi_w^* f\in\cH$, and the linear span of $\{\psi_w^* h_1,\dots,\psi_w^* h_{d+1}\}$ is equal to $\cH$.
Therefore, $\psi_w^* f$ is expressed as a linear combination of $\psi_w^* h_1,\dots,\psi_w^* h_{d+1}$.
This implies that there exists a constant $c_9>0$ such that
$\nu_{\psi_w^* f}\le c_9\sum_{i=1}^{d+1}\nu_{\psi_w^* h_i}$.
For any Borel subset $B$ of $K_w$, from \Lem{energymeas}~(iii),
\[
\nu_f(B)=\frac1{r_w}\nu_{\psi_w^* f}(\psi_w^{-1}(B))
\le \frac{c_9}{r_w}\sum_{i=1}^{d+1}\nu_{\psi_w^* h_i}(\psi_w^{-1}(B))
=c_9\sum_{i=1}^{d+1}\nu_{h_i}(B)
=c_9\nu(B).
\]
Therefore, $\nu_f$ on $K_w$ is absolutely continuous with respect to $\nu$ on $K_w$.
Since $w\in W_m$ is arbitrary, we obtain $\nu_f\ll \nu$.

Since $\cH_*$ is dense in $\cF$, we conclude from \Lem{basic} that $\nu$ is a minimal energy-dominant measure.
\end{proof}
\begin{lemma}\label{lem:existence}
There exists $g\in\cH$ such that $\nu_g$ is a minimal energy-dominant measure.
\end{lemma}
\begin{proof}
The proof is similar to that of \Prop{dense}, and it is simpler.
Set $Z^{i,j}(x)=(d\nu_{h_i,h_j}/d\nu)(x)$ for $i,j\in\{1,\dotsc,d+1\}$, and define $Z(x)=\left(Z^{i,j}(x)\right)_{i,j=1}^{d+1}$, $x\in K$.

Any $h\in\cH$ can be expressed uniquely as $h=\sum_{i=1}^{d+1}a_i h_i$ for some $\bfa={}^t(a_1,\dots,a_{d+1})\in\R^{d+1}$.
Then, 
\[
 \frac{d\nu_h}{d\nu}(x)=\sum_{i,j=1}^{d+1}a_ia_j\frac{d\nu_{h_i,h_j}}{d\nu}(x)
 =(\bfa,Z(x)\bfa)_{\R^{d+1}}
 \quad \nu\mbox{-a.e.\,}x.
\]
Since $Z(x)$ is not a zero matrix for $\nu$-a.e.\,$x$ from the definition of $\nu$, $\ker Z(x)$ is a proper subspace of $\R^{d+1}$.
Therefore, when $\lm$ is the Lebesgue measure on $\R^{d+1}$, for $\nu$-a.e.\,$x$, $(\bfa,Z(x)\bfa)_{\R^{d+1}}>0$ for $\lm$-a.e.\,$\bfa$.
From the Fubini theorem, for $\lm$-a.e.\,$\bfa$, $(\bfa,Z(x)\bfa)_{\R^{d+1}}>0$ for $\nu$-a.e.\,$x$.
In particular, for such $\bfa={}^t(a_1,\dots,a_{d+1})\in\R^{d+1}$, $g:=\sum_{i=1}^{d+1}a_i h_i$ satisfies $(d\nu_g/d\nu)(x)>0$ $\nu$-a.e.\,$x$.
Therefore, $\nu_g$ is a minimal energy-dominant measure.
\end{proof}
\begin{lemma}\label{lem:singularity}
Suppose that two sequences $\{u_n\}_{n=1}^\infty$ and $\{v_n\}_{n=1}^\infty$ in $\cF$ converge $u$ and $v$ in $\cF$, respectively.
If there exist a sequence of Borel sets $\{B_n\}_{n=1}^\infty$ of $K$ such that $\lim_{n\to\infty}\nu_{u_n}(K\setminus B_n)=0$ and $\lim_{n\to\infty}\nu_{v_n}(B_n)=0$, then $\nu_u\perp\nu_v$.
In particular, if $\nu_{u_n}\perp\nu_{v_n}$ for every $n$, then $\nu_u\perp\nu_v$.
\end{lemma}
\begin{proof}
Take a sequence $\{n(m)\}\uparrow\infty$ such that $\cE(v-v_{n(m)})\le2^{-m}$ and $\nu_{v_{n(m)}}(B_{n(m)})\le2^{-m}$ for every $m\in\N$.
Set
\[
  \tilde B_k=\bigcup_{m=k}^\infty B_{n(m)}\text{ for } k\in\N,\quad\mbox{and}\quad \tilde B=\bigcap_{k=1}^\infty \tilde B_k.
\]
Then, for any $k\in\N$ and $m\ge k$,
\begin{align*}
\sqrt{\nu_u(K\setminus\tilde B_k)}
&\le \sqrt{\nu_{u_{n(m)}}(K\setminus\tilde B_k)}
    + \sqrt{\nu_{u-u_{n(m)}}(K\setminus\tilde B_k)}
    \qquad \mbox{(from \Eq{energy})}\\
&\le \sqrt{\nu_{u_{n(m)}}(K\setminus B_{n(m)})}
    + \sqrt{ 2\cE(u-u_{n(m)})}\\
&\to 0 \qquad\text{as }m\to\infty.
\end{align*}
Therefore, $\nu_u(K\setminus\tilde B_k)=0$ for every $k\in\N$, which implies $\nu_u(K\setminus\tilde B)=0$.
On the other hand, for $m\in\N$,
\begin{align*}
\sqrt{\nu_{v}(B_{n(m)})}
&\le \sqrt{\nu_{v_{n(m)}}(B_{n(m)})}
    + \sqrt{\nu_{v-v_{n(m)}}(B_{n(m)})}\\
&\le 2^{-m/2}+ \sqrt{ 2\cE(v-v_{n(m)})}\\
&\le (1+\sqrt2)2^{-m/2}.
\end{align*}
Then, for $k\in\N$,
\[
  \nu_{v}(\tilde B_k)\le \sum_{m=k}^\infty(1+\sqrt2)^2 2^{-m}
  =(1+\sqrt2)^2 2^{1-k}.
\]
Therefore, $\nu_v(\tilde B)=0$.
This concludes that $\nu_u\perp\nu_v$.
\end{proof}
\begin{lemma}\label{lem:singular}
Let $f_1,f_2\in\cF$ and suppose $\nu_{f_1}\perp\nu_{f_2}$.
Then, for any $w\in W_*$, $\nu_{\psi_w^*f_1}\perp\nu_{\psi_w^*f_2}$.
\end{lemma}
\begin{proof}
There exists a Borel set $B$ of $K$ such that $\nu_{f_1}(B)=0$ and $\nu_{f_2}(K\setminus B)=0$.
From \Lem{energymeas}~(iii), 
\[
\nu_{\psi_w^*f_1}(\psi_w^{-1}(B))
=\nu_{\psi_w^*f_1}(\psi_w^{-1}(B\cap K_w))
=r_w\nu_{f_1}(B\cap K_w)=0
\]
 and 
\[
\nu_{\psi_w^*f_2}(K\setminus\psi_w^{-1}(B))
=\nu_{\psi_w^*f_2}(\psi_w^{-1}(K_w\setminus B))
=r_w\nu_{f_2}(K_w\setminus B)=0.
\]
Therefore, $\nu_{\psi_w^*f_1}\perp\nu_{\psi_w^*f_2}$.
\end{proof}
Recall that $\bfone\in l(V_0)$ is a constant function on $V_0$ with value $1$.
Let 
\[
\tilde l(V_0)=\{u\in l(V_0)\mid (u,\bfone)_{l(V_0)}=0\}
\]
 and 
\[
P\colon l(V_0)\ni u\mapsto u-(u,\bfone)_{l(V_0)}\bfone\in l(V_0)
\]
be the orthogonal projection of $l(V_0)$ onto $\tilde l(V_0)$.
We define 
\[
\cH'=\{h\in\cH\mid (\iota^{-1}(h),\bfone)_{l(V_0)} =0,\ \nu_h(K)(=2\cE(h))=1\},
\]
which is a compact subset of $\cF$ that does not include any constant functions.
We define a map $J\colon \cH\setminus\{\mbox{constant functions}\}\to \cH'$ by
\[
  J(h)=\frac{\iota\circ P\circ \iota^{-1}(h)}{\sqrt{2\cE(h)}}.
\]

Let $q\in V_0$. 
From \cite[Theorem~A.1.2]{Ki} (essentially, from the Perron--Frobenius theorem), $A_q$ has simple eigenvalues $1$ and $r$, and the absolute values of any other eigenvalues are less than $r$.
Let $u_q$ be the column vector $\left(D_{q'q}\right)_{q'\in V_0}$, that is,
\begin{equation}\label{eq:u}
  u_q(x)=\left\{\!\!\begin{array}{cl}-d&(x=q)\\1&(x\ne q)\end{array}\right.\!,
  \quad x\in V_0.
\end{equation}
Then, from \cite[Lemma~5]{HN06}, $u_q$ is an eigenvector of ${}^t\! A_q$ with respect to the eigenvalue $r$.
Let $\tilde v_q$ be an eigenvector of $A_q$ with respect to the eigenvalue $r$.
Then, $\tilde v_q(q)$ should be $0$, and other components should be the same by symmetry.
Therefore, $\tilde v_q$ can be taken as
\begin{equation}\label{eq:v}
  \tilde v_q(x)=\left\{\!\!\begin{array}{cl}0&(x=q)\\1&(x\ne q)\end{array}\right.\!,
  \quad x\in V_0.
\end{equation}
Then, $(u_q,\tilde v_q)_{l(V_0)}=d$ and the following holds.
\begin{lemma}[({\cite[Lemma~6]{HN06}})]\label{lem:HN}
  Let $q\in V_0$ and $u\in l(V_0)$.
  Then, 
  \[
  \lim_{n\to\infty}r^{-n}PA_q^n u = \frac{(u_q,u)_{l(V_0)}}{d} P\tilde v_q.
  \]
\end{lemma}
From the lemma above, we have the following.
\begin{lemma}\label{lem:convergent}
  Let $q\in V_0$ and $u\in l(V_0)$.
  Suppose $(u_q,u)_{l(V_0)}\ne0$. Then,
  $\lim_{n\to\infty}J(\iota(A_q^n u))=J(\iota(\tilde v_q))$ in $\cF$.
\end{lemma}
\begin{proof}
From \Lem{HN},
\begin{align*}
J(\iota(A_q^nu))
&=\frac{\iota\circ PA_q^n u}{\sqrt{2\cE(\iota(A_q^n u))}}
=\frac{\iota\circ r^{-n}PA_q^n u}{\sqrt{2\cE(\iota(r^{-n}A_q^n u))}}\\
&\overset{n\to\infty}{\longto}
\frac{\iota(c P\tilde v_q)}{\sqrt{2\cE(\iota(c P\tilde v_q))}} \qquad\left(c:=\frac{(u_q,u)_{l(V_0)}}{d}\ne0\right)\\
&=J(\iota(\tilde v_q)).
\end{align*}
This implies the claim.
\end{proof}
\begin{proof}[of the \Thm{sg}]
Take $g\in\cH\cap \cFz$, which is possible by  \Lem{existence}.
We may assume $\nu_g(K)=1$.
Let $h$ be an arbitrary nonconstant function in $\cH$.
Define $B=\{x\in K\mid (d\nu_h/d\nu_g)(x)=0\}$.
Note that $\nu_h(B)=0$.
In order to prove that $\nu_h$ is also a minimal energy-dominant measure, it suffices to show that $\nu_g(B)=0$.

We will derive a contradiction by assuming $\nu_g(B)>0$. 
For $n\in\Z_+$, let $Y_n$ be the conditional expectation of $1_B$ with respect to $\nu_g$ given $\cB_n$, that is, 
\[
  Y_n(x)=\frac{\nu_g(K_{[x]_n}\cap B)}{\nu_g(K_{[x]_n})}\quad\mbox{for $x\in K\setminus V_*$}.
\]
Then, from the martingale convergence theorem, $Y_n$ converges to $1_B$ $\nu_g$-a.e.
In particular, there exists $y\in B\setminus V_*$ such that $Y_n(y)$ converges to $1$ as $n\to\infty$.
For $n\in\Z_+$, define $g_n=\psi_{[y]_n}^* g$, $h_n=\psi_{[y]_n}^* h$, and $B_n=\psi_{[y]_n}^{-1}(B)$.
From \Lem{energymeas}, $\nu_{g_n}(B_n)/\nu_{g_n}(K)=Y_n(y)\to1$ as $n\to\infty$ and $\nu_{h_n}(B_n)=r_{[y]_n}\nu_h(B\cap K_{[y]_n})=0$. Note that $g_n$ and $h_n$ are both nonconstant harmonic functions for every $n$ from the assumption of the nondegeneracy of $(\cE,\cF)$.
We define $g'_n=J(g_n)$ and $h'_n=J(h_n)$ for $n\in\Z_+$.
Then,
\begin{align*}
\nu_{g'_n}(K\setminus B_n)&=1-\frac{\nu_{g_n}(B_n)}{\nu_{g_n}(K)}\to 0 \quad\text{as }n\to\infty,\\
\nu_{h'_n}(B_n)&=0 \quad \text{for every $n\in\Z_+$}.
\end{align*}
Take a sequence $\{n(l)\}\uparrow\infty$ such that $\{g'_{n(l)}\}_{l=1}^\infty$ and $\{h'_{n(l)}\}_{l=1}^\infty$ converge to some $g'$ and $h'$ in $\cF$, respectively.
Then, $g',h'\in\cH'$, and $\nu_{g'}\perp\nu_{h'}$ from \Lem{singularity}.

Since $Du=0$ if and only if $u$ is constant, and $\iota^{-1}(g')$ is not constant, there exists $i\in\{1,\dotsc,d+1\}$ such that $(u_{q_i},\iota^{-1}(g'))_{l(V_0)}\ne0$.
For $n\in\Z_+$, define
$\hat g_n=J(\psi_{\underbrace{\scriptstyle i\dotsm i}_{n}}^*g')=J(\iota\circ A_{q_i}^n\circ \iota^{-1}(g'))$ and $\hat h_n=J(\psi_{\underbrace{\scriptstyle i\dotsm i}_{n}}^*h')$.
There exists a sequence $\{n(k)\}\uparrow\infty$ such that each $\{\hat g_{n(k)}\}_{k=1}^\infty$ and $\{\hat h_{n(k)}\}_{k=1}^\infty$ converges in $\cF$ as $k\to\infty$.
The limits will be denoted by $\hat g$ and $\hat h$, respectively.
Both limits belong to $\cH'$.
From \Lem{convergent}, $\hat g=J(\iota(\tilde v_{q_i}))$.
Since $\nu_{\hat g_n}\perp\nu_{\hat h_n}$ for each $n\in\Z_+$ from \Lem{singular}, $\nu_{\hat g}\perp \nu_{\hat h}$ from \Lem{singularity}.

Now, $(u_{q'},\iota^{-1}(\hat g))_{l(V_0)}=(u_{q'},\tilde v_{q_i})_{l(V_0)}\left/\sqrt{2\cE(\iota(\tilde v_{q_i}))}\right.\ne0$ for every $q'\in V_0$, which is confirmed by \Eq{u} and \Eq{v}.
Let us assume that $(u_{q_j},\iota^{-1}(\hat h))_{l(V_0)}\ne0$ for some $j\in\{1,\dotsc,d+1\}$.
For $n\in\Z_+$, define
$\hat{\hat g}_n=J(\psi_{\underbrace{\scriptstyle j\dotsm j}_{n}}^*\hat g)=J(\iota\circ A_{q_j}^n\circ \iota^{-1}(\hat g))$ and $\hat{\hat h}_n=J(\psi_{\underbrace{\scriptstyle j\dotsm j}_{n}}^*\hat h)=J(\iota\circ A_{q_j}^n\circ \iota^{-1}(\hat h))$.
Then, from \Lem{convergent}, both $\{\hat{\hat g}_n\}_{n=0}^\infty$ and $\{\hat{\hat h}_n\}_{n=0}^\infty$ converge to $f:=J(\iota(\tilde v_{q_j}))\in \cH'$ in $\cF$. 
Since $\nu_{\hat{\hat g}_n}\perp\nu_{\hat{\hat h}_n}$ for all $n$ from \Lem{singular}, $\nu_f\perp\nu_f$ from \Lem{singularity}.
Since $\nu_f$ is not a null measure, this is a contradiction, which implies that $(u_{q_j},\iota^{-1}(\hat h))_{l(V_0)}=0$ for all $j\in\{1,\dotsc,d+1\}$.
But this means that $D(\iota^{-1}(\hat h))=0$ and $\iota^{-1}(\hat h)$ is constant, which contradicts the fact that $\hat h\in \cH'$.
Consequently, we conclude that $\nu_g(B)=0$ and complete the proof.
\end{proof}
\begin{remark}
\begin{enumerate}
\item For most p.c.f.\ self-similar fractals except Sierpinski gaskets (for example, all fractals shown in Figure~\ref{fig:pcf} except the first three), Dirichlet forms associated with harmonic structures are {\em not} nondegenerate. That is, there exist nonconstant harmonic functions that are constant on some small cells, which means that the energy measures of such functions are not minimal energy-dominant measures. Therefore, in order to obtain the claim in \Thm{sg}, nondegeneracy is a necessary assumption. Thus, considering only Sierpinski gaskets is not as restrictive as it may appear.
\item The key point of the proof above is the fact that $(u_{q'},\tilde v_q)_{l(V_0)}\ne0$ for every $q,q'\in V_0$.
Thus, the assertion of \Thm{sg} is still true when we perturb the harmonic structure $(D,\bfr)$ slightly.
\item
With regard to the Radon--Nikodym derivative $d\nu_{h_1}/d\nu_{h_2}$, some numerical computation has been carried out in \cite{ST08}.
\end{enumerate}
\end{remark}

\affiliationone{%
   Masanori Hino\\
   Graduate School of Informatics\\
   Kyoto University\\
   Kyoto 606-8501\\
   Japan
   \email{hino@i.kyoto-u.ac.jp}}
\end{document}